\documentclass[11pt]{amsart}
\usepackage{amssymb, colordvi}
\usepackage{multirow}
\usepackage[pdftex]{graphicx,color}
\usepackage{tikz}
\usepackage{pgffor}
\usetikzlibrary{arrows,positioning,matrix,decorations.pathmorphing,calc,fadings,decorations.pathreplacing} 
\usepackage[hidelinks]{hyperref}
\usepackage{todonotes} 
\usepackage{enumerate}
\usepackage{mathrsfs}

\numberwithin{equation}{section}
\newtheorem{thm}{Theorem}
\numberwithin{thm}{section}
\newtheorem{prop}[thm]{Proposition}
\newtheorem{lemma}[thm]{Lemma}
\newtheorem{corollary}[thm]{Corollary}
\newtheorem{cor}[thm]{Corollary}

\newtheorem{example}[thm]{Example}
\newtheorem{remark}[thm]{Remark}
\newtheorem{definition}[thm]{Definition}

\newcounter{FNC}[page]
\def\fauxfootnote#1{{\addtocounter{FNC}{2}$^\fnsymbol{FNC}$%
     \let\thefootnote\relax\footnotetext{$^\fnsymbol{FNC}$#1}}}

\newcommand{\Q}{\mathbb{Q}}
\newcommand{\R}{\mathbb{R}}
\newcommand{\Z}{\mathbb{Z}}
\renewcommand{\P}{\mathbb{P}}
\newcommand{\A}{{\mathcal A}}

\newcommand{\rk}{{\rm rk}}
\DeclareMathOperator{\degree}{deg}
\DeclareMathOperator{\sign}{sign}
\DeclareMathOperator{\signvar}{signvar}
\DeclareMathOperator{\ch}{{\rm chull}}

\DeclareMathOperator{\I}{I}
\DeclareMathOperator{\supp}{supp}

\title[Sign conditions for at least one positive solution]{Sign conditions for the existence 
of at least one positive solution of a sparse polynomial system} 

\author{Fr\'ed\'eric Bihan}
\address{Laboratoire de Math\'ematiques\\
         Universit\'e Savoie Mont Blanc\\
         73376 Le Bourget-du-Lac Cedex\\
         France}
\email{Frederic.Bihan@univ-savoie.fr}
\urladdr{http://www.lama.univ-savoie.fr/~bihan/}

\author{Alicia Dickenstein}
\address{Dto.\ de Matem\'atica, FCEN, Universidad de Buenos Aires, and IMAS (UBA-CONICET), Ciudad Universitaria, Pab.\ I, 
C1428EGA Buenos Aires, Argentina}
\email{alidick@dm.uba.ar}
\urladdr{http://mate.dm.uba.ar/~alidick}

\author{Magal\'{\i} Giaroli}
\address{Dto.\ de Matem\'atica, FCEN, Universidad de Buenos Aires, and IMAS (UBA-CONICET), Ciudad Universitaria, Pab.\ I, 
C1428EGA Buenos Aires, Argentina}
\email{mgiaroli@dm.uba.ar}

\thanks{FB is partially supported by the grant ANR-18-CE40-0009 of Agence Nationale de Recherche, France. AD and MG are partially supported by UBACYT 20020170100048BA, 
CONICET PIP 11220150100473, and ANPCyT PICT 2016-0398, Argentina.}

\begin{document}

\begin{abstract}
We give sign conditions on the support and coefficients of a sparse system of $d$ 
generalized polynomials in $d$ variables that guarantee the existence of at least one positive real root, 
based on degree theory and Gale duality. In the case of integer
exponents, we relate our sufficient conditions to
 algebraic
 conditions that emerged in the study of toric ideals.
  \end{abstract}

\maketitle

\section{Introduction}\label{sec:intro}

Deciding whether a real polynomial system has a positive solution is a basic
question, that is decidable via effective elimination of quantifiers~\cite{BPR}. 
There are few results on lower bounds on the number of real or positive roots of polynomial  
systems (see e.g.~\cite{BSS,soprunova,SottileBook,wang}). 
In this paper, we  consider generalized polynomial systems, that is, polynomials with real exponents, for
which the positive solutions are well defined. We give sign conditions on the support and on 
the coefficients of a sparse system of $d$  generalized polynomials in $d$ variables that guarantee the existence 
of at least one positive real root, based on degree theory and Gale duality.

We fix an exponent set $\A=\{a_1,\ldots,a_{n}\} \subset \R^d$ of cardinality
$n$ and for any given real matrix $C =(c_{ij}) \in \R^{d \times n }$ we consider the associated sparse generalized
 multivariate polynomial system in $d$ variables $x=(x_1, \dots, x_d)$ with support $\A$:
\begin{equation}\label{E:system}
f_i(x)=\sum_{j=1}^{n} c_{ij}x^{a_j} = 0 \, , \quad i=1,\ldots,d.
\end{equation}
We are interested by the existence of the positive solutions 
of~\eqref{E:system} in  $\R_{>0}^d$. 
Denoting by $n_\A(C)$ the (possibly infinite) number of positive real solutions  
of the system~\eqref{E:system}, our main goal is to give sufficient conditions on the exponent set $\A$ and the
coefficient matrix $C$  that ensure 
$n_{\mathcal{A}}(C)>0$.

 When $\A \subset \Z^d$ we  consider the existence of solutions in
the real torus $(\R^*)^d$ of points in $\R^d$ with nonzero coordinates, 
and we relate our conditions to well-studied algebraic properties of lattice ideals associated with the configuration
$\A$.

In applications, for example, in the context of chemical reaction networks, lower bounds of 
positive roots of polynomial systems guarantee the existence of (stoichiometrically compatible) 
positive steady states. In~\cite{MFRCSD}, sign conditions are used to decide if 
a family of polynomial systems associated with a given reaction network cannot admit 
more than one positive solution for any choice of the parameters and, in this case, 
conditions for the existence of one positive solution
are given as a corollary of a result from~\cite{MR}, based on degree theory.
Our point of view of searching for conditions on the exponent and the coefficient matrices of the system 
comes from this paper. As we do not assume
injectivity (at most one root), we cannot use tools from these papers or the
more recent article~\cite{MHR}, as Hadamard's theorem.

  In~\cite{CFMW} the authors use degree theory in the study of chemical reaction
  networks to describe parameters for which there is a single positive solution or for which 
  there are more (this is called \emph{multistationarity}).  We apply some of these techniques
   in a Gale duality setting, more precisely, based on Theorem~\ref{th:degree}, 
which is version of a particular case of Theorem~$2$ in the Supplementary Information of~\cite{CFMW}.

We can use different convex sets to 
apply Theorem~\ref{th:degree}. The first one which comes in mind is the positive orthant, which is not bounded.
Another natural idea is to consider the Newton polytope of the polynomials  in the system, or some dilates of it. 
This is reasonable since it is completely determined by the \emph{monomials} appearing in the system. 
In this paper, we use another convex polytope which seems natural since it is determined by \emph{the coefficients}
of the system. We define this polytope (called $\Delta_P$)  in \S~\ref{ssec:deltap}
using the Gale duality trick for polynomial systems that was
studied by Bihan and Sottile in \cite{BS07}, see also \cite{BS08}. We can think of $\Delta_P$ as 
a ``shadow" of the positive orthant via Gale duality, which has the advantage that it can be chosen to be bounded.

In the case of a generalized polynomial $f$ in one variable, 
 the existence of at least one positive real root based on degree theory is just
 the condition that the coefficients of the smallest exponent and the biggest
 exponent in $f$ have opposite signs. In case both coefficients have the same
 sign, this application is silent and other methods, in general ad hoc, are required
 to ensure a positive root. 
 In our case, this is similar. We explicitly show in Example~\ref{ex:circuit} that when we restrict our results to the case
when $\A \subset \R^n$ is a circuit, our sufficient conditions to have
 $n_{\mathcal A}(C)>0$  are indeed equivalent to the fact that the two extremal coefficients in the Descartes type sign variation bound
for $n_\A(C)$  in~\cite{BD} have opposite signs.
 In the general case, we get interesting and more involved sufficient
 sign conditions relating the associated exponent matrix $A$ in~\eqref{eq:overlineA} and the coefficient matrix $C$, 
 but these conditions are not necessary and other tools are needed
 when they do not hold. 

In Section~\ref{sec:back} we recall the notion of Gale duality and the basic 
duality of solutions (see Theorem~\ref{Galesystems}), and we introduce useful notation 
as well as the necessary condition~\eqref{eq:nonempty} of the existence of a positive vector
in ${\rm ker}(C)$ for $n_A(C)$ being positive, which only depends on the signs of the maximal minors of $C$ 
(that is, on its associated {\it oriented matroid}). 

In Section~\ref{sec:main}, we recall the basic concepts of degree theory and we present our main result
Theorem~\ref{basicmainresult}, which gives 
conditions on the Gale duality side to guarantee the existence of positive solutions. 
This result has some technical hypotheses that relate the lattice of faces of $\Delta_P$ plus signs of entries in a Gale dual matrix 
of the coefficient matrix $C$ of the system,  with sign conditions on a Gale dual of the matrix $A$ of exponents. 
We present Example~\ref{example:3positivesolutions} 
to clarify our definitions and to show that the sufficient conditions in Theorem~\ref{basicmainresult}
could imply the existence of more than just one positive solution.

In the following sections we give sufficient conditions on the support and the matrix of coefficients 
that ensure that Theorem~\ref{basicmainresult} can be applied. In Section~\ref{sec:alg}, we  consider 
the notion of mixed dominating matrices from~\cite{FS} (see Definition~\ref{def:dominating})
 to get Theorem~\ref{th:BdominatingCsamesign}. 
Under the hypotheses of this result, we show that $n_A(C)>0$ whenever the vectors of signs of a basis of 
 ${\rm ker}(A)$, can be realized as vectors of signs of elements in ${\rm ker}(C)$, which is
also a condition only depending on the associated oriented matroid of $C$.  
In Section~\ref{sec:geom}, we give geometric conditions  on $\A$ and $C$
 that guarantee  that the hypotheses of Theorem~\ref{basicmainresult} are satisfied. 
 We introduce for this the notion of $I$-compatibility in Definition~ \ref{def:compatible},
which is clarified in Example~\ref{ex:Icompatible} and Figure~\ref{fig:Icompatible}.
 Based on results from~\cite{FS2}, we get Theorem~\ref{th:Icompatible} which guarantees
 $n_A(C)>0$ in terms of the relative combinatorial positions of the configuration of exponents and the
 configuration of the points given by the columns of $C$.

 In the last Section~\ref{sec:integer}, we concentrate our study on 
 integer configurations $\A$. We relate the dominance conditions to algebraic
 conditions that emerged in the study of toric ideals, and we naturally extend
 in this case our approach to ensure the existence of solutions in the real torus $(\R^*)^d$.
 Our last result is Theorem~\ref{thm:positive impliesreal}, which shows which sign conditions can be
 attained by the corresponding Gale dual system~\eqref{homogeneous Gale system}.

\section{Gale duality for positive solutions of polynomial systems}\label{sec:back}

We first present basic definitions and results on Gale duality. 
Given a matrix $M\in {\R^{r\times s}}$ of maximal rank $r$, a \emph{Gale dual matrix} of $M$
 is any matrix $N \in \R^{s \times (s-r)}$ of
maximal rank whose columns vectors are a basis of the kernel of $M$.
Clearly a Gale dual matrix is not unique as it corresponds to a choice of a basis: it is unique up to right multiplication
by an invertible $(s-r)\times (s-r)$-matrix.  
We will also say that the  $s$ row vectors of $N$ define a Gale dual configuration in $\R^{s-r
}$ to the configuration in $\R^r$ defined by the $s$ column vectors of $M$.
We will introduce a Gale dual system~\eqref{homogeneous Gale system}
and polyhedra $\Delta_P$ \eqref{eq:deltaP}, depending on the choice of a Gale dual matrix to the coefficient
matrix $C$. We will then recall Theorem~\ref{Galesystems}, which gives a fundamental link
between the positive real roots of system \eqref{E:system} and the solutions in $\Delta_P$ of the 
Gale dual system~\eqref{homogeneous Gale system}.

\subsection{Matrices and their Gale duals}\label{ssec:matrices}
Let $\A=\{a_1,\ldots,a_{n}\} $ be a finite subset of $\R^d$ of cardinality
$n$ and $C =(c_{ij}) \in \R^{d \times n }$. As we mentioned in the introduction,  
we are interested in the solvability of the associated sparse generalized
 multivariate polynomial system~\eqref{E:system} in $d$ variables 
 $x=(x_1, \dots, x_d)$ with support $\A$ and coefficient matrix $C$.

Note that if we multiply each equation of system~\eqref{E:system} by a  
monomial (i.e, we translate the configuration $\mathcal{A}$), the number of 
positive real solutions does not change, and then $n_\A(C)$ is an 
affine invariant of the point configuration $\mathcal{A}$. It 
is then natural to consider the matrix $A\in\R^{(d+1)\times n}$ 
with columns $(1,a_1),(1,a_2),\ldots,(1,a_{n})\in\R^{d+1}:$ 
\begin{equation}\label{eq:overlineA} A=\begin{pmatrix}
1 & \dots & 1\\
a_1 & \dots & a_{n}
\end{pmatrix}.
\end{equation}
We will refer to the matrix $A$ as the corresponding matrix of the point configuration $\mathcal{A}$. 

 We will always assume that $C$ is of maximal rank $d$ and $A$ is of maximal rank $d+1$.
 Then, we need to have  $n \ge d+1$. If equality holds,  it is easy to see that 
 system~\eqref{E:system} has a positive
 solution if and only if the necessary condition~\eqref{eq:nonempty} holds.
 So we will suppose that $n\ge d+2$.
 
We denote by $k=n-d-1$ the codimension of $A$ (and of $\A$).  Note that the codimension of $C$ equals
$k+1$. Let $B=(b_{ij})\in\R^{n\times k}$ be a matrix which is Gale dual to $A$, and let 
$D=(d_{ij})\in\R^{n \times (k+1)}$ be any a matrix which is Gale dual to $C$. 
We will number the \emph{columns} of $B$ from 1 to $k$ and the \emph{columns} of $D$ from $0$ to $k$ and
denote by $P_1, \dots, P_{n} \in \R^{k+1}$ the \emph{row} vectors of $D$, that is, the Gale dual configuration to
the columns of $C$.

\subsection{A necessary condition}\label{ssec:neccond}
There is a basic necessary condition for $n_\A(C)$ to be positive.
Denote by $C_1, \dots, C_{n} \in \R^d$  the \emph{column} vectors of the coefficient matrix $C$ and call
\begin{equation}
{\mathcal C}^\circ = \R_{>0} C_1 + \dots + \R_{>0} C_{n},
\end{equation}
the positive cone generated by them.
Given a solution $x \in \R_{>0}^d$ of system \eqref{E:system},  
the vector $(x^{a_1}, \dots, x^{a_{n}})$  is positive and so
the origin ${\bf 0} \in \R^d$ belongs to $\mathcal C^\circ$.
Then, necessarily
\begin{equation}\label{eq:nonempty}
{\bf 0} \in {\mathcal C}^\circ.
\end{equation}
 It is a well-known result that  Condition~\eqref{eq:nonempty} holds if and only if the vectors
 $P_1, \dots, P_{n}$ lie in an open halfspace through the origin. 
 Indeed, any vector in the kernel of $C$ is of the form $(\langle P_1, u \rangle, \dots, \langle P_n, u\rangle)$,
  and so there exists a positive vector in the kernel if and only there exists a vector $u\in \R^{k+1}$
   such that $\langle P_i, u\rangle >0$ for any $i$.
 Note that Condition~\eqref{eq:nonempty}, 
 together with the hypothesis that $C$ is of maximal rank $d$, is equivalent to ${\mathcal C}^\circ 
=\R^d$.

\subsection{Defining cones and polytopes in Gale dual space}\label{ssec:deltap}
We also define other cones that we will use. Denote by
\begin{equation}\label{eq:poscone}
{\mathcal C}_P=\R_{>0}P_1+\cdots+\R_{>0}P_{n},
\end{equation}
the positive cone generated by the rows of a Gale dual matrix $D$ and let 
\begin{equation}\label{eq:Cnu}
{\mathcal C}_P^{\nu}=\{y \in \R^{k+1} \, : \, \langle P_i,y \rangle >0, i=1,\ldots, n\},
\end{equation}
 be its dual open cone. Note that if $C$ has maximal rank $d$ and Condition~\eqref{eq:nonempty} holds, the cone
${\mathcal C}_P$ is strictly convex. Therefore, its dual open cone ${\mathcal C}_P^{\nu}$ 
 is a nonempty full dimensional open convex cone.  We will also consider the closed cone 
\begin{equation}\label{eq:overlineCP}
\overline{{\mathcal C}}_P=\R_{ \geq 0}P_1+\cdots+\R_{\geq 0}P_{n}.
\end{equation}
 
The following Lemma is straightforward.
\begin{lemma}\label{lemma: can be made bounded}
Assume that $C$ has maximal rank $d$ and that ${\bf 0}\in\mathcal C^\circ$. Then for any nonzero
$u \in \overline{{\mathcal C}}_P$ and any $c\in\R_{>0}$, the polytope
${\mathcal C}_P^{\nu} \cap \{y \in \R^{k+1}\, : \, \langle u,y \rangle =c\}$ has dimension $k$. 
Moreover, this polytope is bounded if and only $u \in {\mathcal C}_P$.
\end{lemma}

Define
\begin{equation}\label{eq:deltaP}\Delta_P={\mathcal C}_P^{\nu} \cap \{y \in \R^{k+1}\, : \, y_0=1\}.
\end{equation}

\begin{cor}\label{bounded}
Assume that $C$ has maximal rank, ${\bf 0}\in\mathcal C^\circ$ and let $D$ be a Gale dual matrix of $C$. Then
$(1,0,\ldots,0)\in {\mathcal C}_{P}$  if and only if $\Delta_P$ has dimension $k$ and is bounded.
\end{cor}

We next show that we can always find 
a Gale matrix $D$ such that $\Delta_P$ is nonempty and bounded.

\begin{lemma}\label{lemma:inCp}
	Assume that $C$ has maximal rank.
%
	Then there is a Gale dual matrix $D$ of $C$ such that $(1,0,\ldots,0)\in {\mathcal C}_{P}$.
\end{lemma}
\begin{proof}
	Start with any Gale dual matrix $D$ of $C$ and pick any vector $u \in {\mathcal C}_{P}$.
	Then there is an invertible matrix $R \in \R^{(k+1) \times (k+1)}$ such that $u \cdot R =(1,0,\ldots,0)$, where $u$ is 
written as a row vector.
	Consider the matrix $D'=DR$ and denote by $P_1',\ldots,P_{n}'$ its row vectors.
	Then $D'$ is Gale dual to $C$, and $(1,0,..,0) \in {\mathcal C}_{P'}=\R_{>0}P'_1+\cdots+\R_{>0}P'_{n}$. 
\end{proof}

To any choice of Gale dual matrices $B$ and $D$ of $A$ and $C$ respectively, we associate the following system
with unknowns $y=(y_0,\ldots,y_k)$:
\begin{equation}\label{homogeneous Gale system}
\prod_{i=1}^{n} \langle P_i, y \rangle^{b_{ij}}=1,\; j=1,\ldots,k,
\end{equation}
which is called a \emph{Gale dual system} of \eqref{E:system}. Denote $G_j(y)=\prod_{i=1}^{n} \langle P_i, y \rangle^{b_{ij}}$.
Another choice $D'$ of a Gale dual matrix for $C$ corresponds to another choice $y'$ of linear coordinates for $\R^{k+1}$:
if $D'= D R$ with $R \in \R^{(k+1)\times(k+1)}$ invertible, then setting $y'=R^{-1}(y)$ 
we get $D' y'=D y$, where $y$ and $y'$ as considered as column vectors.
Another choice $B'$ of a Gale dual matrix for $A$ gives an equivalent Gale system $H_1=\dots=H_k=1$, where for each $j$  there
exist exponents $(\mu_1, \dots, \mu_k)$ such that $H_j = G_1^{\mu_1} \dots G_k^{\mu_k}$.

Note that \eqref{homogeneous Gale system} is homogeneous of degree zero since the columns of $B$ sum up to zero.
For any cone ${\mathcal C} \in \R^n$ with apex the origin, its projectivization $\P{\mathcal C}$ is the quotient space ${\mathcal C} / \sim$
under the equivalence relation $\sim$ defined by: for all $y,y' \in {\mathcal C}$, 
we have  $y \sim y'$ if and only if there exists $\alpha>0$
such that $y=\alpha y'$.

We will often use the following observation.

\begin{remark}\label{proj}
If $(1,0,\ldots,0) \in \overline{{\mathcal C}}_P$, then ${\mathcal C}_P^{\nu}$ is contained
 in the open half-space defined by $y_0>0$. Thus, the map
$(y_0,y_1,\ldots,y_k) \mapsto (1,y_1/y_0,\ldots,y_k/y_0)$ induces a bijection between $\P{\mathcal C}_P^{\nu}$ and $\Delta_P$.
\end{remark}

\subsection{The equivalence of solutions}\label{ssec:eqsol}
Here is a slight variation of  Theorem 2.2 in~\cite{BS07}.
\begin{thm}\label{Galesystems}
There is a bijection between the positive solutions of the initial system~\eqref{E:system}
and the solutions of the Gale dual system~\eqref{homogeneous Gale system} in $\P{\mathcal C}_P^{\nu}$,
 which induces a bijection between the positive solutions of
~\eqref{E:system} and the solutions of~\eqref{homogeneous Gale system} in $\Delta_P$ 
when $(1,0,\ldots,0) \in \overline{{\mathcal C}}_P$.
\end{thm}
\begin{proof}
If $x \in \R_{>0}^d$ is a solution of the system~\eqref{E:system}, then $(x^{a_1},\ldots,x^{a_{n}})$
 belongs to $\ker(C) \cap \R_{>0}^{n}$. Thus, there exists $y\in \R^{k+1}$ (which is unique since $D$ 
 has maximal rank) such that $x^{a_i}=\langle P_i, y \rangle$ for $i=1,\ldots,n$. Then,  $y \in {\mathcal C}_P^{\nu}$ 
 and $y$ is a solution of the Gale dual system~\eqref{homogeneous Gale system}. If furthermore 
 $(1,0,\ldots,0) \in \overline{{\mathcal C}}_P$, then dividing by $y_0$ if necessary, a solution 
 $y \in {\mathcal C}_P^{\nu}$ of \eqref{homogeneous Gale system} gives a solution of the same 
 system in $\Delta_P$ because it is homogeneous of degree zero. We showed in Remark~\ref{proj} 
 that the previous map is bijective by giving explicitly its inverse map.
 
Now, let $y \in {\mathcal C}_P^{\nu}$ be a solution of \eqref{homogeneous Gale system}.
Let $(e_1,\ldots,e_d)$ be the canonical basis of $\R^d$.
Since $A$ has maximal rank, there exists $\alpha_j=(\alpha_{1j},\ldots,\alpha_{nj}) \in\R^{n}$, for $j=1,\ldots,d$, 
such that $e_j=\sum_{i=1}^{n} \alpha_{ij}a_i$. To any column vector $z \in \R^{k+1}$, we associate the 
vector $D \cdot z$ with coordinates $\langle P_i ,z \rangle$, $i=1,\ldots,n$. Consider now the map 
\begin{equation*}
\begin{array}{rl}\varphi\colon\R^{k+1}&\to \R^{d}\\ z &\mapsto
\left((D \cdot z)^{\alpha_1},\ldots, (D \cdot z)^{\alpha_d}\right), \end{array}
\end{equation*}%
where $(D \cdot z)^{\alpha_j}=\prod_{i=1}^{n} \langle P_i,z \rangle^{\alpha_{ij}}$.
Let $x=\varphi(y)$. Then, $ x^{a_i}=\langle P_i,y \rangle$ for $i=1,\ldots,n$, which gives 
$(x^{a_1},\ldots,x^{a_{n}})\in \ker(C)$. Moreover, since $y \in {\mathcal C}_P^{\nu}$, 
we have that $x \in \R_{>0}^d$, and then $x$ is a positive solution of system~\eqref{E:system}.
\end{proof}

\begin{remark}
Theorem 2.2 in \cite{BS07} is a particular case of Theorem \ref{Galesystems} taking a Gale dual matrix 
$D$ with the identity matrix $I_{k+1}$ as a submatrix (in which case the condition that 
$(1,0,\ldots,0) \in \overline{{\mathcal C}}_P$ is trivially satisfied).
\end{remark}

\section{Existence of positive solutions via Gale duality and degree theory}\label{sec:main}

In this section, we present Theorem~$\ref{basicmainresult}$, which gives conditions on the Gale dual matrices $B$ and $D$ that 
guarantee the existence of at least one positive solution of the system~\eqref{E:system}. As we mentioned in the Introduction, 
our results are based on degree theory. We will consider coefficient matrices $C$ of maximal rank $d$ such that  the necessary
condition~\eqref{eq:nonempty} is satisfied.  As before, we will fix a matrix 
$D=(d_{ij})\in\R^{n \times (k+1)}$ which is Gale dual to $C$ and 
denote by $P_1, \dots, P_{n} \in \R^{k+1}$ the row vectors of $D$, which are all nonzero.
We will moreover assume that
the polytope $\Delta_P$ in ~\eqref{eq:deltaP}  is nonempty and bounded.
Our main result is Theorem~\ref{basicmainresult}.  We also point out how the statement
simplifies in case $C$ is uniform, that is, when all maximal minors of $C$ are nonzero and we give
examples that show how this result can be applied.

Given an open set $U\subset \R^k$, a function $h\in \mathcal{C}^0(\overline{U},\R^k)$ and 
$y\in \R^k\setminus h({\partial} U)$, the symbol $\degree(h,U,y)$ denotes the Brouwer degree (which belongs to $\Z$) of $h$ with
respect to $(U, y)$. A main result in degree theory is that if $\degree(h,U,y)\neq 0$, then there 
exists at least one $x\in U$ such that $y = h(x)$. For background and the main properties about 
Brouwer degree, we refer to Section 2 in the Supplementary Information of~\cite{CFMW} and Section 14.2 in \cite{teschl}.

We present the version of the Brouwer's theorem that we will use. This version is a particular case 
of Theorem $2$ in the Supplementary Information of~\cite{CFMW} (here we take $W$ empty), 
and also appears in the proof of Lemma 2 of \cite{monotone}. 
Recall that a vector $v\in \R^k$ \emph{points inwards} $U\subset \R^k$ at a boundary point $x \in \partial U$,
 if for small  $\varepsilon >0$ it holds that
$x + \varepsilon v \in \overline{U}$.

 \begin{thm}[\cite{CFMW,monotone}] \label{th:degree}
Let $h: \R^k\to \R^k$ be a $\mathcal{C}^1$-function. Let $U$ be an open, nonempty, bounded and convex subset 
of $\R^k$ such that
\begin{enumerate}
\item[i)] $h(x)\neq 0$ for any $x\in \partial U$.
\item[ii)] for every  $x\in \partial U$, the vector $h(x)$ points inwards $U$ at $x$.
\end{enumerate}
Then, \[\degree(h,U,0)=(-1)^k.\]
In particular, there exists a point $x$ in $U$ such that $h(x)=0$. Moreover, assuming the zeros are nondegenerate, if there exists 
a zero $x^*\in U$ where the Jacobian at $x^*$ has the same sign as $(-1)^{k+1}$, then there are at least three zeros and always an odd number.
\end{thm}

Define the sign of any real number $r$ by
$\sign(r)=+1,-1, 0$ according as $r>0$, $r<0$ or $r=0$ respectively.
The sign of any vector $r=(r_1,\ldots,r_k) \in \R^k$ is then defined by $\sign(r)=(\sign(r_1),\ldots,\sign(r_k))$.

In view of Theorem~\ref{Galesystems}, we look for the solutions of \eqref{homogeneous Gale system} 
in $\Delta_P$. Plugging $y_0=1$ in \eqref{homogeneous Gale system} and clearing the denominators, 
we get a  generalized polynomial system in $\Delta_P$ on variables $y=(y_1,\ldots,y_k)$:
\begin{equation}\label{nothomogeneous Gale system}
g_j(y)=0 , \quad  j=1,\ldots,k,
\quad g_j(y)=\prod_{b_{ij}>0} p_i(y) ^{b_{ij}}-\prod_{b_{ij}<0} p_i(y) ^{-b_{ij}},\end{equation}
where 
 \begin{equation}\label{eq:pi}
 p_i(y)=\langle P_i, (1,y) \rangle.
 \end{equation}
 We denote by $g$ the Gale map:
\begin{equation}\label{eq:G}
g=(g_1,\dots,g_k)\colon \R^k\to\R^k.
\end{equation}

Recall that Condition~\eqref{eq:nonempty} implies that all $P_i$ lie 
in an open halfspace through the origin.
\begin{definition}\label{def:ICbis} Given $C\in\R^{d\times n}$ of maximal rank $d$ satisfying condition~\eqref{eq:nonempty}, 
we denote by $\bar{I}_C \subset\{1,\dots,n\}$ the set of indices
of the vectors $P_{i}$ which belong to the boundary of the cone $\mathcal{C}_P$ and by $I_C \subset \bar{I}_C$ the set of indices
of the vectors $P_i$ which lie on the rays of the cone (that is, on its one-dimensional faces).
\end{definition}

As the facets of ${\mathcal C}_{P}^{\nu}$ are supported on the orthogonal
hyperplanes $P_i^{\bot}$ for  $i \in I_C$, 
 for any $i \in I_C$ the vector $P_i$ is an inward normal vector of 
${\mathcal C}_{P}^{\nu}$ at any point in the relative interior
of the facet supported on $P_i^{\bot}$. It follows that the facets of the polytope $\Delta_P$ are supported on the hyperplanes 
$p_i(y)=0$ for $i \in I_C$, and that $(d_{i1},\ldots, d_{ik})$ is an inward normal vector 
of $\Delta_P$ at any point in the relative interior of the 
facet $F_i$. Also,  we denote by $F_i$ the face of $\Delta_P$ supported on $p_i(y)=0$ for any $i\in \bar{I}_C$,
and for any $L \subset \bar{I}_C$ we set $F_L=\cap_{i \in L} F_i$. Here, by a face of $\Delta_P$ 
we mean a face of the closure of $\Delta_P$. We denote by $F_L^{\circ}$
the relative interior of $F_L$. 

\begin{definition}\label{Face maximal}
Let $C\in\R^{d\times n}$ of maximal rank $d$ satisfying condition~\eqref{eq:nonempty}, with Gale dual $D$.
For a given proper  face $F$ of $\Delta_{P}$, we denote by $L(F) = \{ i \in \bar{I}_C \, : \, F \subset F_i \}$
  the maximal set of indices $L \subset \bar{I}_{C}$ such that $F=F_{L}$ and set
\[\mathcal F(\Delta_P)=\{L(F)\, : \, F\, \mbox{is a proper face of} \ \Delta_P\}\]
\end{definition}
Note  that $\bar{I}_C$ and $I_C$ depend on $C$ and are independent of the choice of  Gale dual matrix $D$. 
Moreover, in case $C$ is uniform $D$ is also uniform, 
$\bar{I}_C= I_C$ and the cardinality of $L(F)$ equals the codimension of the face, which means that the
polytope $\Delta_P$ is simple.

The sign of $g$ along the boundary of $\Delta_P$
can sometimes be determined as follows.

\begin{lemma}\label{sign of G} 
Let $A\in\R^{(d+1)\times n}$ as in \eqref{eq:overlineA} of maximal rank $d+1$, $C\in \R^{d\times n}$ 
a uniform matrix satisfying the necessary condition~\eqref{eq:nonempty}, and let  $B\in\R^{n\times k}$ and 
$D\in\R^{n\times (k+1)}$ be  Gale dual matrices of $A$ and $C$ 
respectively. Let $g=(g_1,\dots,g_k)$ the Gale map as in~\eqref{eq:G}. Let $j \in \{1,\ldots, k\}$.
\begin{enumerate}
\item Let $F_i$ be any facet of $\Delta_P$ and let $x \in F_i^{\circ}$. If $b_{ij} \neq 0$, then $\sign(g_j(x))=-\sign(b_{ij})$.
\item
Let $L\in \mathcal F(\Delta_P)$ and $x \in F_L^{\circ}$.
Assume that 
$\{b_{\ell j} \, : \,  \ell \in L\} \neq \{0\}$.
\begin{enumerate}[(i)]
\item If there exists $\ell_0,\ell_1 \in L$ such that $b_{\ell_0 j} \cdot b_{\ell_1 j}<0$, then  $g_j(x)=0$.
\item If $b_{\ell j} \geq 0$ for all $\ell \in L$ then $\sign(g_j(x))=-1$, and if $b_{\ell j} \leq 0$ for all $\ell \in L$ then 
$\sign(g_j(x))=+1$.
\end{enumerate} 
\end{enumerate}
\end{lemma}

\begin{proof}  Since $B$ is a Gale dual matrix of $A$, then each column of $B$ contains at least 
a positive and a negative entry, that is, the sets $\{b_{\ell j}>0 \, : \ell=1,\dots,n \}$ 
and $\{b_{\ell j}<0 \, : \ell=1,\dots,n \}$ are nonempty for each $j=1,\dots,k$.

For the first item, note that the facet $F_i$ is supported on $p_i(y)=0$. Suppose $x \in F_i^{\circ}$ 
and $b_{ij} \neq 0$. We compute $g_j$ as in~\eqref{nothomogeneous Gale system}. If $b_{ij}>0$, then 
\[ g_j(x)=-\prod_{b_{\ell j}<0} p_\ell (x) ^{-b_{\ell j}},\]
because $p_i(x)=0$. In this case, $g_j(x)<0$ because as $C$ is uniform, $D$ is also uniform and so no other $p_\ell$
can be a nonzero multiple of $p_i$.  Then, $\sign(g_j(x))=-\sign(b_{ij})$. The case $b_{ij}<0$ is analogous. 

For the second part, assume that $L\in \mathcal F(\Delta_P)$ and $x \in F_L^{\circ}$, 
with $\{b_{\ell j} \, : \,  \ell \in L\} \neq \{0\}$. If there exists $\ell_0,\ell_1 \in L$ such 
that $b_{\ell_0 j} \cdot b_{\ell_1 j}<0$, we have that $p_{\ell_1}(x)=p_{\ell_2}(x)=0$, and then both terms
$\prod_{b_{\ell j}>0} p_\ell (x) ^{b_{\ell j}}$ and  $\prod_{b_{\ell j}<0} p_\ell (x) ^{-b_{\ell j}}$ are zero. Then $g_j(x)=0$. 
In the other case, if $b_{\ell j} \geq 0$ for all $\ell \in L$, then $b_{\ell j} > 0$ for some $\ell \in L$ because we assume
that not all vanish, and thus
\[ g_j(x)=-\prod_{b_{\ell j}<0} p_\ell (x) ^{-b_{\ell j}}.\]
Note that if $x \in F_L^{\circ}$, $p_\ell(x)=0$ if and only if $\ell\in L$. Then, since $b_{\ell j} \geq 0$
 for all $\ell \in L$, we have that $p_{\ell}(x) \neq 0$ for all $\ell$ such that $b_{\ell j} <0$, 
 and so the product $-\prod_{b_{\ell j}<0} p_\ell (x) ^{-b_{\ell j}}$ is negative. 
 The case where $b_{\ell j} \leq 0$ for all $\ell \in L$ is analogous. 
\end{proof}

\begin{corollary}\label{coro:Gvanish}
Let $A\in\R^{(d+1)\times n}$ as in \eqref{eq:overlineA} of maximal rank $d+1$, $C\in \R^{d\times n}$ a uniform matrix of rank $d$
satisfying condition~\eqref{eq:nonempty},  and $B\in\R^{n\times k}$ 
and $D\in\R^{n\times (k+1)}$ Gale dual matrices of $A$ and $C$ respectively. Let $g$ be the Gale map~\eqref{eq:G} 
associated to $B$ and $D$. If $g(x)=0$ and $x \in F_L^{\circ}$ (so $L \in  \mathcal 
F(\Delta_P)$),
then for $j=1,\ldots,k$, either $\{b_{\ell j} \, :  \, \ell \in L\}=\{0\}$, or $\{b_{\ell j} \, : \, \ell \in L\}$
contains a (strictly) positive and a (strictly) negative element.
In particular, if $g$ vanishes in the relative interior of a facet $F_{\ell}$ then the $\ell$-th row of $B$ contains only zero entries.
\end{corollary}

In order to state Theorem~\ref{basicmainresult} we need to introduce  the following definitions.
\begin{definition}\label{def:facesbis}
Let $C\in\R^{d\times n}$ of maximal rank $d$. 
We define an equivalence relation on $\{1, \dots, n\}$ where two indices $i_1,i_2$ are
equivalent whenever  $P_{i_1}$ and $P_{i_2}$ are proportional. We can then partition
$\bar{I}_{C }$ into equivalence classes $K_1, \dots, K_s$.

Given $B\in \R^{n \times k}$ we define an associated matrix $\bar{B} \in \R^{s \times k}$
with the following coefficients. For any $r =1, \dots, s$ and any $j=1,\ldots,k$ set
\begin{equation}\label{eq:sum}
\bar{b}_{rj}:=\sum_{\ell \in K_{r}} b_{\ell j}.
\end{equation}
 For any $r \in \{1, \dots, s\}$ choose an index $i(r)$ such that $i(r) \in K_r$. Given any $L\subset \{1,\dots,n\}$  we set  
 $$\bar{L} \, =\, \{r \in \{1,\dots,s\} \; : \; L \cap K_r \neq \emptyset \}$$
 and denote by $\bar{B}_{\bar{L}}$ the submatrix of $\bar{B}$ given by the rows with indices in $\bar{L}$.
\end{definition}

Note that when $C$ satisfies~\eqref{eq:nonempty},
if $P_{\ell}$ and $P_{i}$ are equivalent,  there exists a positive constant $c$ such that
$P_i = c P_\ell$. If moreover $C$ is uniform, each equivalence class $K_i$ consists of a single index (that we can assume to be $i$)
and  $\bar{b}_{ij} = b_{ij}$ for any $i$ in $\bar{I}_C =I_C$ and any $j\in \{0, \dots, k\}$.

\begin{definition}\label{def:weaklymixed} We say that a matrix $M$ is \emph{weakly mixed} if any column of $M$ either has only 
zero entries, or contains a positive and a negative element. 
\end{definition}

 We now present the main result of this section.

\begin{thm}\label{basicmainresult}
Let $\mathcal{A}=\{a_1,\dots,a_{n}\}\subset \R^d$  and $C\in\R^{d\times n}$ of maximal rank $d$.
Let $A\in\R^{(d+1)\times n}$ as in  \eqref{eq:overlineA} of maximal rank $d+1$, and $B\in\R^{n\times k}$ and 
$D\in\R^{n\times (k+1)}$ Gale dual matrices of $A$ and $C$ respectively.  Assume 
that ${\bf 0}\in\mathcal C^\circ$ and that  $\Delta_P$ is a full dimensional bounded polytope.

Assume furthermore that the following conditions hold:
\begin{enumerate}
\item For any $L \in \mathcal F(\Delta_P)$ the submatrix $\bar{B}_{\bar{L}}$ is not weakly mixed, that is, 
$\bar{B}_{\bar L}$ has a nonzero column whose entries are all either nonpositive, or nonnegative.
%
\item For any $i\in \bar{I}_C$ as in Definition~\ref{def:ICbis}, and $r$ such that $i\in K_r$, the following holds: 
\begin{itemize}
\item $\bar{b}_{rj} \cdot d_{ij} \geq 0$ for any $j=1,\ldots,k$,
\item there exists $j\in \{1,\ldots,k\}$ such that  $\bar{b}_{rj} \cdot d_{ij} > 0$,
\item for all $j \in \{1,\ldots,k\}$, if  $\bar{b}_{rj}=0$ then $d_{ij}=0$.
\end{itemize}
\end{enumerate}
Then $n_{\mathcal A}(C) >0$.
\end{thm}

\begin{proof}
In view of Theorem~\ref{Galesystems}, we look for the solutions of \eqref{homogeneous Gale system} 
in $\Delta_P$, given by $G_{j}(y)=1$ for $j=1,\ldots,k$.
For each  $r \in \{1,\ldots,s\}$, choose one representative $i(r)$ in the equivalence class $K_r$.
Under the assumption ${\bf 0}\in\mathcal C^\circ$, if $P_{i}$ and $P_{j}$ are proportional
 then $P_{j}=c \cdot P_{i}$ with $c >0$, as noticed before. Therefore,  there exist constants $c_{j}>0$ for $j=1,\ldots,k$ 
 such that \eqref{homogeneous Gale system} is equivalent to
\begin{equation}\label{homogeneous Gale systembis}
\prod_{r \in \{1,\ldots,s\}} \langle P_{i(r)}, y \rangle^{\bar{b}_{r,j}}=c_{j},\; j=1,\ldots,k.
\end{equation}
Plugging $y_0=1$ and clearing the denominators,  system~\eqref{homogeneous Gale system}  has
the same solutions as the following  generalized polynomial system in $\Delta_P$ on variables $y=(y_1,\ldots,y_k)$:
\begin{equation}\label{nothomogeneous Gale systembis}
\bar{g}_{j}(y)=0 , \quad  j=1,\ldots,k,
\end{equation}
where
\begin{equation}\label{E:general}
\bar{g}_{j}(y)=\prod_{\bar{b}_{r,j}>0} p_{i(r)}(y)^{\bar{b}_{r,j}}- c_{j }\prod_{\bar{b}_{r,j}<0} p_{i(r)}(y) ^{-\bar{b}_{r,j}}.
\end{equation}

Note that by construction, if $r,r' \in  \{1,\ldots, s\}$ are distinct, then $P_{i(r)}$ and $P_{i(r')}$ are not proportional.
Thus system \eqref{nothomogeneous Gale systembis} looks like a system \eqref{nothomogeneous Gale system} with the
exception given by the positive constants $c_{j}$. It is then not difficult to see that,  the conclusions of 
Lemma~\ref{sign of G} and Corollary \ref{coro:Gvanish} hold for the map $\bar{g}=(\bar{g}_1,\ldots,\bar{g}_{k})$ 
corresponding to the system \eqref{nothomogeneous Gale systembis}. 
To simplify the notation, we assume from now on that $C$ is uniform, and so, as we remarked before, $I_C=\bar{I}_C$ and  $\bar{b}_{ij}=b_{ij}$.

Since $\Delta_P$ is full dimensional and bounded, $(1,0,\dots,0)\in\mathcal{C}_P$. 
To prove that $n_{\mathcal A}(C) >0$, by Theorem~\ref{Galesystems} it is sufficient to show that the Gale system
 \eqref{nothomogeneous Gale systembis} has at least one 
solution in $\Delta_P$.
First note that a vector $v \in \R^k$ points inwards $\Delta_P$ at a point $y$ contained in 
the relative interior of the facet $F_i$  supported on $P_i^{\bot}$
with $i \in I_C$
if and only if $\langle (d_{i1},\ldots,d_{ik}), v \rangle \geq 0$. More generally  
$v \in \R^k$ points inwards $\Delta_P$ at a point $y$ in the 
relative interior of a face
$F_L$  ($L \in \mathcal F(\Delta_P)$) if and only if $\langle (d_{\ell 1},\ldots,d_{\ell k}),v \rangle \geq 0$ for any $\ell \in L$, by a classical 
result of convex geometry.
Assumption (1) ensures that $g$ does not vanish at $\partial \Delta_P$, by Corollary \ref{coro:Gvanish}.
By Lemma \ref{sign of G},  condition (2) ensures that $-g$ points inwards $\Delta_P$ at each point $x$ in the relative 
interior of any facet $F_i$. Then $-g$ also points inwards $\Delta_P$ at any point $x$ in the relative interior of a face $F_L$.
The result follows now from Theorem~\ref{th:degree}, taking $U=\Delta_P$ and $h=-g$.
\end{proof}

\begin{example}\label{ex:circuit}
Consider the codimension one case $k=1$ (which is treated carefully in \cite{BD}). Then $B \in \R^{(d+2) \times 1}$ is a column matrix and
its entries are the coefficients $\lambda_1,\ldots,\lambda_{d+2}$ of a nontrivial affine relation on $\mathcal A$.
Assume that $A$ is uniform (equivalently, assume that $\mathcal A$ is a circuit \footnote{A point configuration 
$\mathcal{A}$ of $d+2$ points is a \emph{circuit} if any subset of $d+1$ points of $\mathcal{A}$ is affinely independent.}). 
Then, $B$ has no zero entry.
Assume moreover that $C$ is uniform and that ${\bf 0}\in\mathcal C^\circ$. Then, there exists a Gale dual matrix $D$ such that
$\Delta_P$ is a bounded interval of $\R$. Moreover, there exists a vector $\delta \in \R^2$ 
such that $\langle  P_i, \delta \rangle >0$
for $i=1,\ldots,d+2$, where $P_1,\ldots,P_{d+2} \in \R^{2}$ are the row vectors of $D$. 
Let $\alpha: \{1,\ldots,d+2\} \rightarrow \{1,\ldots,d+2\}$ be the bijection
such that all determinants $\mbox{det} \, (P_{\alpha_{i}}, P_{\alpha_{i+1}})$ are positive for $i=1,\ldots,d+1$.
Then, by Theorem~2.9 in~\cite{BD}, we have $n_{\mathcal A}(C) \leq \signvar(\lambda_{\alpha_{1}},
\lambda_{\alpha_{2}},\ldots, \lambda_{\alpha_{d+2}})$  and moreover the difference is an even 
integer number (see Proposition~2.12 in~\cite{BD}).
The endpoints of the interval $\Delta_P$ are the roots of the two extremal polynomials $p_{\alpha_{1}}$ 
and $p_{\alpha_{d+2}}$, equivalently, $I_C=\{\alpha_{1}, \alpha_{d+2}\}$.
Now the Gale polynomial $g=g_{1}:\R \rightarrow \R$ (or its opposite) points inwards $\Delta_P$
at its vertices if and only if $\lambda_{\alpha_1} \cdot \lambda_{\alpha_{d+2}}<0$, which is equivalent to
$\signvar(\lambda_{\alpha_1}, \lambda_{\alpha_{d+2}})=1$. But,  $\signvar(\lambda_{\alpha_1},
\lambda_{\alpha_2},\ldots, \lambda_{\alpha_{d+2}})$ and $\signvar(\lambda_{\alpha_1}, \lambda_{\alpha_{d+2}})$ have the same parity.
Thus $g:\R \rightarrow \R$ (or its opposite) points inwards $\Delta_P$ at its vertices 
if and only if $n_{\mathcal A}(C)$ is odd by Proposition~2.12 in~\cite{BD}.
Therefore, in the circuit case the sufficient condition to have $n_{\mathcal A}(C)>0$ which is given by Theorem~\ref{basicmainresult}
is equivalent to $n_{\mathcal A}(C)$ being odd. Now, for any integer $d \geq 2$, it is not difficult
 to get examples of circuits ${\mathcal A} \subset \R^d$ and
matrices $C$ such that $n_{\mathcal A}(C)$ is odd and is different from $1$. This shows that our sufficient condition
does not imply $n_{\mathcal A}(C)=1$ in general, and thus is not equivalent to the condition given 
in \cite{MFRCSD} ensuring that $n_{\mathcal A}(C)=1$.
\end{example}

We now present an example with $k=d=2$ to illustrate Theorem~\ref{basicmainresult}.

\begin{example}\label{example:3positivesolutions}

Let $\mathcal{A} \subset \Z^2$ be the set of points $a_1=(0,4)$, $a_2=(5,4)$, $a_3= (2,8)$, $a_4=(3,0)$
and $a_5=(3,5)$. Consider 
the matrix of coefficients
\[C=\left(\begin{array}{ccccc}
-1 & -1 & 1 & 1 & 0\\  
-(3c+8) & -c & 2c+8 & 0 & 2
\end{array} \right), \]
where $c\in\R$ is a parameter.
The polynomial system of two polynomial equations and two variables $x, y$:
\[\begin{array}{rcl}
	- y^4 - x^5 y^4 + x^2 y^8 + x^3 & = & 0,\\
	-(3c+8)y^4 - c x^5 y^4 + (2c + 8) x^2 y^8 + 2 x^3 y^5 & = & 0,
	\end{array}\]
has support $\mathcal{A}$ and coefficient matrix $C$.
Let $A$ as in \eqref{eq:overlineA}.
Choose the following Gale dual matrices of $A$ and $C$ :

\[B=\left(\begin{array}{rr}
 1 & 0 \\
 2 & 1\\
 1 & 2\\
0 & 1\\
-4 & -4
\end{array} \right) \qquad D= \left(\begin{array}{r|rr}
1 &  1 & 0 \\ 
1 & 1 & 2 \\
1 & 2 & 1 \\
1 & 0 & 1 \\  
c & -4 & -4 
\end{array} \right) \]
Then $p_1(y)=1+y_1$, $p_2(y)=1+y_1+2y_2$, $p_3(y)=1+2y_1+y_2$, $p_4(y)=1+y_2$ and $p_5(y)=c-4y_1-4y_2$. 
If $c>0$, the convex polytope $\Delta_P$ is nonempty, bounded and it has five facets supported 
on the lines $p_i=0$ for $i=1,\ldots,5$, see Figure~\ref{fig:deltap3}.
Moreover, if $c>0$, then the assumptions of Theorem~\ref{basicmainresult} are satisfied and thus $n_{\mathcal A}(C)>0$.
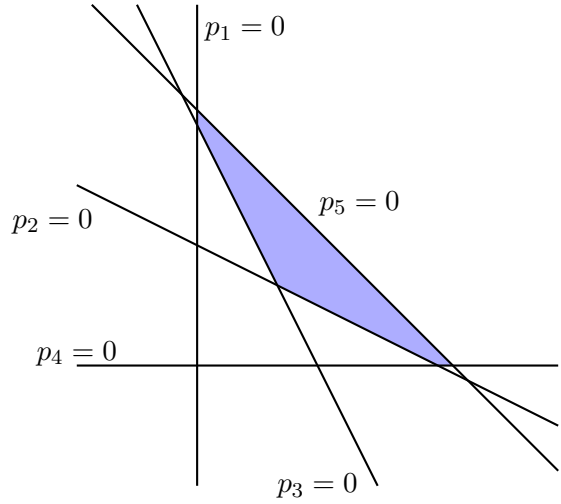
\begin{figure}[h]
	\centering
	\begin{tikzpicture}
	[scale=1.600000,
	back/.style={loosely dotted, thin},
	edge/.style={color=black, thick},
	facet/.style={fill=red!70!white,fill opacity=0.800000},
	vertex/.style={inner sep=1.5pt,circle,draw=black,fill=black,thick,anchor=base
	}]
	%
	%
	\coordinate (0.00000, 0.00000) at (0.00000, 0.00000);
	\coordinate (0.00000, 1.00000) at (0.00000, 1.00000);
	\coordinate (2.00000, 0.00000) at (2.00000, 0.00000);
	\coordinate (2.00000, 1.00000) at (2.00000, 1.00000);
	\coordinate (1.00000, 2.00000) at (1.00000, 2.00000);
	\coordinate (1.00000, 3.00000) at (1.00000, 3.00000);
    \fill[facet, fill=blue!40!white,fill opacity=0.800000] (-1.0000, 1.00000) --  (-1.00000, 1.125000)-- (1.125,-1) 
--(1,-1)--(-0.33333,-0.33333) -- cycle {};
%
	\draw[edge] (-2.00000, 0.50000) -- (2.00000, -1.50000);
	\draw[edge] (0.50000, -2.00000) -- (-1.50000, 2.0000);
	\draw[edge] (-2.0000, -1.00000) -- (2.00, -1.00000);
	\draw[edge] (-1.87500, 2.000) -- (2, -1.87500);
		\draw[edge] (-1.00000, -2.0000) -- (-1.00, 2.00000);
	\node[] at (-2.200, 0.2000){$p_2=0$};
	\node[] at (-0.6000, 1.80000){$p_1=0$};
	\node[] at (0.350, 0.350){$p_5=0$};
	\node[] at (-2.000, -0.90000){$p_4=0$};
	\node[] at (0.000,-2.000){$p_3=0$};
	\end{tikzpicture}
	\caption{The polytope $\Delta_P$ of Example~\ref{example:3positivesolutions}, with $c>0$.}\label{fig:deltap3}
\end{figure}

We use Singular~\cite{Singular}, a free software of computer algebra system, to check what happens when we vary the value of 
$c>0$.

\begin{small}
\begin{verbatim}
%LIB "signcond.lib";
ring r=0, (c,x,y,t), dp;
poly f1=-y^4-x^5*y^4+x^2*y^8+x^3;
poly f2=-(3*c+8)*y^4-c*x^5*y^4+(2*c+8)*x^2*y^8+2*x^3*y^5;
ideal i=f1,f2, diff(f1,x)*diff(f2,y)-diff(f1,y)*diff(f2,x),x*y*t-1;
ideal j=std(i);
ideal k =eliminate(j, x*y*t);
k;
k[1]=48c12+1280c11+12288c10+49152c9+65467c8-2560c7-24576c6-
98304c5-131078c4+1280c3+12288c2+49152c+65563
\end{verbatim}
\end{small}%

The roots of this last polynomial in $c$ correspond to systems with a degenerate solution, and we can check that the only 
positive root of $f_1, f_2$ and their jacobian is $1$. We check, again using Singular~\cite{Singular}, with the library ``signcond.lib" (implemented 
by E. Tobis, based on the algorithms described in~\cite{BPR}) that if we take for example $c=\frac{1}{2}$ ($c<1$), the system has 
3 positive solutions, and if we take $c=\frac{8}{7}$ ($c>1$), the system has only $1$ positive solution. 
We use the command {\small {\tt firstoct}}, that computes the number of roots of a system in the first octant, that is, the
positive roots.

\begin{small}
\begin{verbatim}
LIB "signcond.lib";
ring r=0, (x,y), dp;
poly f1=-y^4-x^5*y^4+x^2*y^8+x^3;
poly f2=-(3*(1/2)+8)*y^4-(1/2)*x^5*y^4+(2*(1/2)+8)*x^2*y^8+2*x^3*y^5;
poly f3=-(3*(8/7)+8)*y^4-(8/7)*x^5*y^4+(2*(8/7)+8)*x^2*y^8+2*x^3*y^5;
ideal i2 = f1,f2;
ideal j2 = std(i2);
firstoct(j2);
3
ideal i3 = f1,f3;
ideal j3 = std(i3);
firstoct(j3);
1
\end{verbatim}
\end{small}%
For $c=\frac{1}{2}$, the condition in~\cite{MFRCSD} to ensure exactly one positive solution is trivially not 
satisfied (as expected since the system has $3$ positive solutions). 

Observe that this procedure is symbolic and thus certified. This computation of the number of positive solutions with 
  the command  {\small {\tt firstoct}} works for
moderately sized polynomial systems with coefficients in $\mathbb Q$ or an algebraic extension of it. Our results
are particularly useful to study families of polynomials.
\end{example}

\section{Dominating matrices} \label{sec:alg}

In this section, we present some conditions on $A$ and $C$ that guarantee that the hypotheses of Theorem 
\ref{basicmainresult} are satisfied. Our main result is Theorem~\ref{th:BdominatingCsamesign}.

We first present conditions that guarantee that a matrix $A$ admits a choice of a 
Gale dual matrix $B$, which satisfies condition (1) of Theorem \ref{basicmainresult}, for \emph{any} uniform matrix of 
coefficients $C$ satisfying~\eqref{eq:nonempty} (which means that it does not depend on $I_C$). When $A\in\Z^{(d+1)\times n}$ 
we will relate these conditions with complete intersection lattice ideals in Section~\ref{sec:integer}.

We recall some definitions from \cite{FS}, with the difference that we replace 
rows by columns and allow matrices with real entries.

\begin{definition}\label{def:dominating}
  A vector is said to be mixed if contains a strictly positive and a strictly negative coordinate.
More generally, a 
real matrix is called mixed if every column contains a strictly 
positive and a strictly negative entry. A real matrix is called dominating if 
it contains no square mixed submatrix. An empty matrix is mixed and also dominating. 
\end{definition}

For example any matrix with the following sign pattern
\begin{small}
\[\begin{pmatrix} + & + & + \\ + & 0 & - \\ + & - & 0 \\ - & 0 & 0  \end{pmatrix}\]
\end{small}%
is mixed and dominating. When a matrix is mixed and dominating we will say that it is \textit{mixed dominating}.


Observe that since a matrix $A$ as in 
\eqref{eq:overlineA} has a row of 
ones, 
the columns of any Gale dual matrix $B$ add up to zero, and thus $B$ is always mixed. 
Also note that a mixed matrix is weakly mixed 
(see Definition~\ref{def:weaklymixed}), but the converse is not true in general as a weakly 
mixed matrix can also contain a column with only zero entries.

We now show that for uniform matrices $A$ admitting a dominating Gale dual matrix $B$, 
condition (1) in Theorem~\ref{basicmainresult} is automatically satisfied for any uniform matrix $C$ having a positive vector
 in the kernel.
 
\begin{lemma}\label{lemma:nominormixed}
Assume that $A\in\R^{(d+1)\times n}$ is a uniform matrix. If $B\in\R^{n\times k}$ is a Gale dual 
matrix of $A$ which is dominating, then condition (1) of Theorem~\ref{basicmainresult}
 is satisfied for all $C\in\R^{d\times n}$ uniform satisfying ${\bf 0}\in\mathcal C^\circ$. 
\end{lemma}
\begin{proof}
Let $C\in\R^{d\times n}$ uniform, and take any Gale dual matrix $D\in\R^{n\times(k+1)}$ 
of $C$ such that $\Delta_P$ is nonempty and bounded (which exists due to  Lemma \ref{lemma:inCp}
and Corollary \ref{bounded}).

If $L=\{\ell\} \in \mathcal F(\Delta_P)$, then 
$B_L$ 
weakly mixed means that it has only zeros, which forces the matrix $A$ minus the $\ell$-th column
to have rank $< d+1$.
Consider $L \in \mathcal F(\Delta_P)$ such that $|L| \geq 2$. Note that $|L| \leq k$ since $C$ is uniform.
If $B_L$ is weakly mixed then at least $k-|L|+1$ columns of $B_L$ contain only zero entries, for otherwise we would get a square 
submatrix of size $|L|\times |L|$
containing a positive and a negative coefficient in each column. But if $k-|L|+1$ columns of $B_L$ contain only zeros, then we 
get $k-|L|+1$ linearly independent vectors in the kernel of
the matrix $A_{\setminus L}$ obtained from $A$ by removing the columns indexed by $L$. 
This forces the matrix $A_{\setminus L}$ to have rank less than $d+1$, which is a contradiction, since $A$ is uniform.
\end{proof}

The following results will be useful. These propositions are only stated for matrices with integer 
coefficients in \cite{FS}, but clearly the proofs given in that paper also work for real matrices.

\begin{prop}[\cite{FS}, Corollary 2.7 and 2.8]\label{prop:mixeddominatingli} If a real matrix is mixed dominating, then any nonzero linear 
combination of its columns is a mixed vector. In particular, its columns are linearly independent.
\end{prop}

The proof of this proposition is a consequence of Proposition 2.6 in~\cite{FS}, which implies that the set 
of indices for the positive entries of any nonzero linear combination of the columns of a dominating matrix contains the set of indices of the positive or negative
entries of some column of the matrix. Moreover, they show that the following is true.

\begin{prop}[\cite{FS2}, Proposition 4.1]\label{prop:mixeddominating0} The left kernel of any mixed 
dominating real matrix contains a positive vector.
\end{prop}

%
%

We will also need the following Lemma.

\begin{lemma}
\label{addpositive vector}
Assume that $C\in\R^{d\times n}$ has maximal rank $d$ and that ${\bf 0}\in\mathcal C^\circ$.
Let $\tilde{D} \in \R^{n \times k}$ be any matrix of maximal rank $k$ such that $C \tilde{D}=0$. Assume that
\begin{equation}\label{condbis}
{\bf 0} \in \R_{>0}\tilde{P}_1+\cdots+\R_{>0}\tilde{P}_{n},
\end{equation}
where $\tilde{P}_1,\ldots, \tilde{P}_{n}$ are the row vectors of $\tilde{D}$.
Then, there exists a positive vector $D_0$ in the kernel of $C$ which does not belong to the 
linear span of the column vectors of $\tilde{D}$, and
the matrix $D \in \R^{n \times (k+1)}$ obtained from $\tilde{D}$ by adding $D_0$
as a first column vector is Gale dual to $C$ and satisfies $(1,0,\ldots,0)\in {\mathcal C}_{P}$.
\end{lemma}
\begin{proof}
By \eqref{condbis} there exists a positive vector in the left kernel of $\tilde{D}$, in other words, a row vector $\lambda$ with positive coordinates
\footnote{In fact, it is sufficient that $\lambda$ is a nonzero vector with only nonnegative coordinates.}
such that $\lambda \cdot \tilde{D}=(0,\ldots,0)$.
Since ${\bf 0} \in \mathcal C^\circ$ we have $\ker(C) \cap \R_{>0}^{n} \neq \emptyset$. Then, as
$\ker(C)$ has dimension $k+1$ and $\tilde{D}$ has rank $k$, there exists a vector 
$D_0 \in \ker(C)\cap \R_{>0}^{n}$ which does not belong to the linear span of 
the column vectors of $\tilde{D}$. The matrix $D \in \R^{n \times k}$ obtained from 
$\tilde{D}$ by adding $D_0$ as a first column vector is Gale dual to $C$. Moreover, we 
have $\lambda \cdot D=( \lambda \cdot D_0, 0, \ldots, 0)$ and thus $(1,0,\ldots,0) \in  {\mathcal C}_{P}$ 
since $\lambda \cdot D_0 >0$ (here $\lambda$ is a row vector, $D_0$ is a column vector so that 
$\lambda \cdot D_0$ is a real number, which is positive since $\lambda$ and
$D_0$ are positive vectors).
\end{proof}

If $S\subset\R^n$ if a subspace, we denote $\sign(S)=\{\sign(v)\, : \, v\in S\}$.
Recall that we denote the column vectors of a matrix $B$ by $B_1, \dots, B_k$. 

\begin{thm}\label{th:BdominatingCsamesign}  Let $\mathcal{A}=\{a_1,\dots,a_{n}\}\subset \R^d$.
Assume $A\in\R^{(d+1)\times n}$ as in~\eqref{eq:overlineA} and  $C\in\R^{d\times n}$ are uniform matrices. 
Suppose there exist a dominating Gale dual matrix $B\in\R^{n\times k}$ of $A$. 
Assume ${\bf 0}\in {\mathcal 
C}^{\circ}$ and $\sign(B_j)\in \sign(\ker (C))$ for each $j=1,\dots, k$. Then, $n_{\mathcal{A}}(C)>0$.
\end{thm}
\begin{proof} As $B$ is dominating
 and $A,C$ are uniform, condition (1) of Theorem~\ref{basicmainresult} is satisfied by 
Lemma~\ref{lemma:nominormixed}. 
As $\sign(B_j)\in \sign(\ker(C))$ for $j=1,\dots, k$, there exist vectors  $D_1,\dots, D_k$ in $\ker(C)$ such that 
$\sign(D_j)=\sign (B_j)$ for each $j=1,\dots,k$. Consider the matrix $\tilde{D}$ with column vectors $D_1,\ldots,D_k$.
Since $B$ is mixed dominating (it is mixed since $A$ contains a row of ones) the matrix $\tilde{D}$
is mixed dominating and furthermore $\tilde{D}$ has rank $k$ by Proposition~ \ref{prop:mixeddominatingli}. 
Moreover, by Proposition~\ref{prop:mixeddominating0}, there is a positive vector
in the left kernel of $\tilde{D}$. Then, condition \eqref{condbis} is satisfied, and thus by 
Lemma \ref{addpositive vector} and Corollary \ref{bounded},
there is a positive vector $D_0$ such that the matrix $D$ with column vectors $D_0,\ldots,D_k$ is Gale dual to $C$
and the associated polytope $\Delta_P$ is nonempty and bounded. By construction, 
condition (2) of Theorem~\ref{basicmainresult} is also satisfied, and thus $n_{\mathcal{A}}(C)>0$.
\end{proof}

Recall that the support of a vector $v\in\R^n$ is defined to be the set of its nonzero coordinates, and we denote it by $\supp(v)$. 
Given a subspace $S\subset\R^n$, a circuit of $S$  is a nonzero element $s\in S$ with minimal support (with respect to inclusion).
 Given a vector $v$, a circuit $s=(s_1,\dots,s_n)$ is said to be conformal to $v=(v_1,\dots,v_n)$ if for any index $i$ in $\supp(s)$, 
  $\sign(s_i) = \sign(v_i)$. The next lemma shows that if $A$ admits a Gale dual mixed dominating matrix, then there exist a choice 
  of Gale mixed dominating matrix of $A$ whose columns are circuits of $\ker(A)$. Note that all the circuits of $\ker(A)$ can be described 
  in terms of vectors of maximal minors of $A$, and so they only depend on the \emph{associated oriented matroid} of $A$.

\begin{lemma}\label{lemma:dominatingcircuits} Assume $A\in\R^{(d+1)\times n}$ as in \eqref{eq:overlineA}. 
Suppose there exist a dominating Gale dual matrix $B\in\R^{n\times k}$ of $A$. 
Then, there exists a dominating Gale dual matrix $B'\in\R^{n\times k}$ of $A$ such that every column of $B'$ is a circuit of $\ker(A)$.
\end{lemma}
\begin{proof} It is a known result that every vector in $\ker(A)$ can be written as a nonnegative sum of 
circuits conformal to it (see~ \cite{rockafellar}). In particular, for every vector in $\ker(A)$, there exists a 
circuit conformal to it. For each column $B_i$ of $B$, $i=1,\dots,k$, take a circuit $B'_i$ of $\ker(A)$ such
 that $B'_i$ is conformal to $B_i$. Now, we take $B'$ the matrix with columns $B'_1,\dots,B'_k$. Every column of $B'$ 
 is a circuit of $\ker(A)$, $B'$ is mixed since $A$ has a row of ones, and is dominating because $B'_i$ is
  conformal to $B_i$ for each $i=1,\dots,k$ and the matrix $B$ is dominating. Since $B'$ is mixed dominating, 
  the columns of $B'$ are linearly independent by Proposition~\ref{prop:mixeddominatingli}, and then $B'$ is a Gale dual matrix of $A$.
\end{proof}

\section{Geometric conditions on $A$ and $C$}\label{sec:geom}

The main result of this section is Theorem~\ref{th:Icompatible}, where we give geometric conditions on $A$ and $C$ that guarantee 
that the hypotheses of Theorem~\ref{basicmainresult} are satisfied.

A characterization of matrices $A$ admitting a mixed dominating Gale dual matrix $B$ can be found in~\cite{FS2}. 
Recall that our definition of mixed dominating matrix differs from the one in \cite{FS2} by replacing rows by columns. 
Here we present this result with our notation. We denote the convex hull of a point configuration $\mathcal{A}$ by $\ch(\mathcal{A})$. 
Recall also that we assume $n \geq d+2$, so that $\mathcal A$ cannot be the set of vertices of a $d$-dimensional simplex. 

\begin{thm}[\cite{FS2}, Theorem 5.6 ]\label{teo5.6}
Let $\mathcal{A}=\{a_1,\dots,a_n\}\subset\R^d$, with $n\geq d+2$ and $A\in\R^{(d+1)\times n}$ as in~\eqref{eq:overlineA}.
 Then $A$ admits a mixed dominating Gale dual matrix $B$ if and only if
$\mathcal A$ can be written as a disjoint union $\mathcal A={\mathcal A}_{1} \sqcup {\mathcal A}_{2}$ such that
\begin{enumerate}
\item the polytopes $\ch({\mathcal A}_{1})$ and $\ch({\mathcal A}_{2})$
intersect in exactly one point,
\item the corresponding matrices $A_1$ and $A_{2}$ as in \eqref{eq:overlineA} of $\mathcal{A}_1$ and 
$\mathcal{A}_2$ respectively, admit mixed dominating Gale dual matrices, and
\item $\dim \ch({\mathcal A})= \dim \ch({\mathcal A}_{1})+ \dim \ch({\mathcal A}_{2})$.
\end{enumerate}
\end{thm}

Moreover, we have:

\begin{lemma} [\cite{FS2}, Corollary 5.7]
If $A$ admits a mixed dominating Gale dual matrix $B$ then $\ch({\mathcal A})$ has at most $2d$ vertices.
\end{lemma}

In particular, by Lemma~\ref{lemma:nominormixed}, we have:
 
\begin{corollary}\label{lemma:pointsinchullofacircuitorsimplexmoregenral} Let $\mathcal{A}=\{a_1,\dots,a_n\}\subset\R^d$. 
Assume that $A$ as in \eqref{eq:overlineA} is uniform and that $\mathcal{A} \subset\R^d$ can be decomposed as a disjoint union 
$\mathcal A={\mathcal A}_{1} \sqcup {\mathcal A}_{2}$ such that conditions (1), (2) and (3) of Theorem~\ref{teo5.6} hold.
Then, there exists a Gale dual matrix $B\in\Z^{n\times k}$ of $A$ such that condition (1) of Theorem~\ref{basicmainresult} is satisfied.
\end{corollary}

The following observation says that if we have a point configuration $\mathcal{A}_v\subset\R^d$ 
such that the corresponding matrix $A_v$ admits a Gale dual mixed dominating matrix, then,  for any other 
point configuration $\mathcal{A}\subset\R^d$ that contains $\mathcal{A}_v$ and their convex hulls 
$\ch (\mathcal{A}), \ch (\mathcal{A}_v)$ coincide
 (that is, $\mathcal{A}$ can be obtained from 
$\mathcal{A}_v$ adding points inside the convex hull), the corresponding matrix $A$ also admits a 
Gale dual mixed dominating matrix.

\begin{lemma}\label{lemma:pointsinchullofacircuitbis} Let $\mathcal{A}=\{a_1,\dots,a_n\}$, 
$\mathcal{A}_v\subset \R^d$ be two point configurations such that $\mathcal{A}_v\subset\mathcal{A}$. 
Assume that the corresponding matrix $A\in\R^{(d+1)\times n}$ 
is uniform and that the following conditions hold:
\begin{enumerate}
\item $\ch (\mathcal{A})=\ch (\mathcal{A}_v)$
\item The corresponding matrix $A_v \in \R^{(d+1) \times |A_v|}$ has a Gale dual matrix $B_v$ which is dominating.
\end{enumerate}
Then, there exists a a mixed dominating Gale dual matrix $B\in\Z^{n\times k}$ of $A$ and thus condition (1) 
of Theorem~\ref{basicmainresult} is satisfied.\end{lemma}

Note that Lemma~\ref{lemma:pointsinchullofacircuitbis} follows from applying several times 
Theorem~\ref{teo5.6} (taking one point from ${\mathcal A}_{v}$ as ${\mathcal A}_{2}$), but 
we present a constructive proof. 

\begin{proof}[Proof of Lemma~\ref{lemma:pointsinchullofacircuitbis}]
Without loss of generality, we may assume that $\mathcal{A}_v=\{a_1,\ldots,a_{s}\}$, with $s \geq d$.
For $i=s+1,\ldots,n$, there exists a subset $\mathcal{A}_i $ of $\mathcal{A}_v$ such that $\mathcal{A}_i $ 
is the set of vertices of a $d$-simplex and $a_i$ is contained in the interior of $\ch(\mathcal{A}_i)$. Then there exists 
an affine relation on $\mathcal{A}_i \cup \{a_i\}$ where the coefficient of $a_i$ is equal to one and the coefficients 
of the points of $\mathcal{A}_i $ are all negative. Using the affine relations on $\mathcal{A}_v$ given by the column vectors
of $B_v$, we get $k$ linearly independent vectors in the kernel of $A$ which are the column vectors of a upper triangular 
block matrix $B$ Gale dual to $A$ of the following form:
\[
B=\left(
\begin{array}{cc}
B_v & R \\
0 & I_{n-s}
\end{array}
\right)
,\]
where $R$ has only nonpositive entries (and at least two negative entries in each column) and $ I_{n-s}$ is the identity matrix of size $n-s$.
Clearly, if $B_v$ is dominating then $B$ is dominating and thus by Lemma~\ref{lemma:nominormixed} 
the first item of Theorem~\ref{basicmainresult} is satisfied.
\end{proof}

We have the following corollary.

\begin{corollary}\label{lemma:pointsinchullofacircuitorsimplex} Let $\mathcal{A}=\{a_1,\dots,a_n\}$, 
$\mathcal{A}_v\subset \R^d$ be two point configurations such that $\mathcal{A}_v\subset\mathcal{A}$ 
and $\ch (\mathcal{A})=\ch (\mathcal{A}_v)$. Assume that the corresponding matrix $A$ is uniform and that 
$\mathcal{A}_v $ is either the set of vertices of $d$-simplex, or a circuit in $\R^d$. Then, there exists a Gale dual 
matrix $B\in\Z^{n\times k}$ of $A$ such that condition (1) 
of Theorem~\ref{basicmainresult} is 
satisfied.
\end{corollary}

Consider $\mathcal{A}$ and the point configuration $\mathscr{C}=\{C_1,\dots,C_{n}\}$ given by the columns of the 
coefficient matrix $C$. Consider the $(d+1)\times n$-matrix
\[\bar{C}=\begin{pmatrix} 1 & \cdots & 1\\ & C & \end{pmatrix}\]
and assume that $A$ and $\bar{C}$ are uniform.
Given a subset $J\subset\{1,\dots,n\}$, we denote $\mathcal{A}_J=\{a_j\, : \, j\in J\}$.

\begin{definition}\label{def:compatible}
Given a subset $I\subset \{1,\dots,n\}$, we say that $A$ and $C$ are \emph{$I$-compatible} if the following conditions hold:
\begin{enumerate} 
\item The corresponding matrices $A_I$ and $\overline{C}_I$ admit Gale dual matrices 
which are mixed dominating and have the same sign pattern,
\item $\ch(\mathcal{A}_I)=\ch(\mathcal{A})$ and $\ch(\mathscr{C}_I)=\ch(\mathscr{C})$,
\item For each $j\notin I$, there exist $J\subset I$, with $|J|=d+1$, such that $a_{j}\in \ch(\mathcal{A}_{J})$ and 
$C_j\in \ch({\mathscr{C}}_{J})$.
\end{enumerate}
\end{definition}

The condition that $A$ and $C$ are \emph{$I$-compatible} can be translated in terms of signs of maximal minors of $A$ 
and $\bar{C}$. Also note that the configurations $\mathcal{A}$ and $\mathscr{C}$ may have different oriented matroids.
The following Example~\ref{ex:Icompatible} shows two $I$-compatible configurations with different oriented matroids.

\begin{example}\label{ex:Icompatible} We show in Figure~\ref{fig:Icompatible} 
an example of two point configurations, $\mathcal{A}=\{a_1,\dots,a_6\}$ 
and ${\mathscr{C}}=\{C_1,\dots,C_6\}$ with $d=2$ and $k=3$, which are $I$-compatible, for $I=\{1,2,3,4\}$.
In this case 
$a_5\in ch(\mathcal{A}_{I_5})$, $C_5\in ch({\mathscr{C}}_{I_5})$ for $I_5=\{1,3,4\}$ and $a_6\in ch(\mathcal{A}_{I_6})$, 
$C_6\in ch({\mathscr{C}}_{I_6})$ for $I_6=\{1,2,3\}$.

\end{example}

\begin{figure}[h]\label{fig:Icompatible}
	\centering
	\begin{tikzpicture}
	[scale=1.600000,
	back/.style={loosely dotted, thin},
	edge/.style={color=black, thick},
	facet/.style={fill=red!70!white,fill opacity=0.800000},
	vertex/.style={inner sep=1pt,circle,draw=black,fill=blue,thick,anchor=base
	}]
	%
	%
	\coordinate (0.00000, 0.00000) at (0.00000, 0.00000);
	\coordinate (0.00000, 1.00000) at (0.00000, 1.00000);
	\coordinate (2.00000, 0.00000) at (2.00000, 0.00000);
	\coordinate (2.00000, 1.00000) at (2.00000, 1.00000);
	\coordinate (1.00000, 2.00000) at (1.00000, 2.00000);
	\coordinate (1.00000, 3.00000) at (1.00000, 3.00000);
%
	\draw[edge] (0,0) -- (0.5,1);
	\draw[edge] (0.5,1) -- (1.5,1.5);
	\draw[edge] (1.5,1.5) -- (2.5,0);
	\draw[edge] (2.5,0) -- (0,0);
	\draw[edge,dashed] (1.5,1.5) -- (0,0);
	\draw[edge,dashed] (0.5,1) -- (2.5,0);
	\draw[edge] (4,0) -- (6.5,0);
	\draw[edge] (6.5,0) -- (6,1.5);
	\draw[edge] (6,1.5) -- (4.5,1.5);
	\draw[edge] (4.5,1.5) -- (4,0);
	\draw[edge,dashed] (4,0) -- (6,1.5);
	\draw[edge,dashed] (4.5,1.5) -- (6.5,0);
	
	\node[vertex] at (0.5,1){};
	\node[vertex] at (1.5,1.5){};
	\node[vertex] at (2.5,0){};
	\node[vertex] at (0,0){};
	\node[vertex] at (1,0.5){};
	\node[vertex] at (1.5,1){};
	\node[] at (0.3,1.1){$a_1$};
	\node[] at (1.7,1.6){$a_2$};
	\node[] at (-0.2,0){$a_4$};
	\node[] at (2.7,0){$a_3$};
	\node[] at (0.9,0.4){$a_5$};
	\node[] at (1.4,0.9){$a_6$};
	\node[vertex] at (4.5,1.5){};
	\node[vertex] at (5.25,1.25){};
	\node[vertex] at (5.25,0.5){};
	\node[vertex] at (6,1.5){};
	\node[vertex] at (6.5,0){};
	\node[vertex] at (4,0){};
	
	\node[] at (4.3,1.6){$C_1$};
	\node[] at (5.05,1.3){$C_5$};
	\node[] at (5.25,0.3){$C_6$};
	\node[] at (6.2,1.6){$C_4$};
	\node[] at (6.7,0){$C_3$};
	\node[] at (3.8,0){$C_2$};

	\end{tikzpicture}
	\caption{$A$ and $C$ are $I$-compatible for $I=\{1,2,3,4\}$.}
\end{figure}
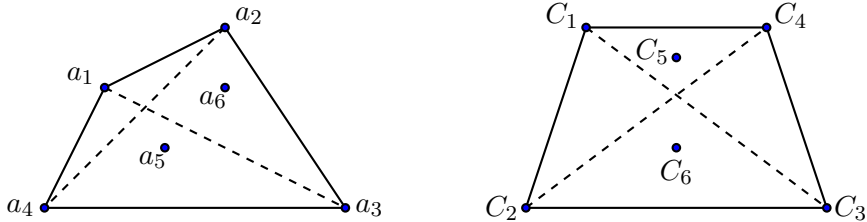


We  have the following result:

\begin{thm}\label{th:Icompatible} Assume that $A$, $C$ and $\bar{C}$ are uniform.
Suppose ${\bf 0}\in {\mathcal C}^{\circ}$, 
and  there exists $I\subset \{1,\dots,n\}$ such that $A$ and $C$ are $I$-compatible. Then, 
$n_{\mathcal{A}}(C)>0$.
\end{thm}

\begin{proof} Let $B_I$ be a Gale dual matrix of $A_I$ as in Condition (1) of Definition~\ref{def:compatible}.
As $a_{j}\in \ch(\mathcal{A}_I)$ for each $j\notin I$, we can use Lemma~\ref{lemma:pointsinchullofacircuitbis}. 
We construct a dominating matrix $B$, using the matrix $B_I$ and using for each $a_{j}$, 
$j\notin I$, the affine relation given by the circuit $a_{j}\cup\mathcal{A}_{J}$, with $J$ as in Condition (3) 
of Definition~\ref{def:compatible}, to obtain a vector in the kernel of $A$ as in the proof of Lemma~\ref{lemma:pointsinchullofacircuitbis}. 
Conditions (1) and (3) of the definition of being $I$-compatible mean that there exist $k$ vectors in the kernel of $\bar{C}$ 
with the same sign patterns as the columns of the constructed $B$, and these $k$ 
vectors are linearly independent because they form a mixed dominating matrix
 (Proposition~\ref{prop:mixeddominatingli}). We have that $\ker{(\bar{C})}\subset \ker(C)$, and
$\sign(B_1),\dots,\sign(B_k)\in \sign(\ker (C))$. We can apply Theorem~\ref{th:BdominatingCsamesign} and then 
$n_{\mathcal{A}}(C)>0$.
\end{proof}

\begin{remark} When $|I|=d+2$, condition (1) in Definition~\ref{def:compatible} can be translated in terms of signatures of circuits. 
Given a circuit $\mathcal{U}=\{u_1,\dots,u_{d+2}\}\subset\R^d$, and a nonzero affine relation $\lambda\in\R^{d+2}$ among the 
$u_i$, 
we call $\Lambda_{+}=\{i\in \{1,\dots,d+2\} \, : \, \lambda_i>0\}$ and $\Lambda_{-}=\{i\in \{1,\dots,d+2\} \, : \, 
\lambda_i<0\}$. The pair $(\Lambda_{+}, \Lambda_{-})$ is usually called a \emph{signature} of $\mathcal{U}$. As $\mathcal{U}$ is a 
circuit, the pairs $(\Lambda_{+}, \Lambda_{-})$ and $(\Lambda_{-}, \Lambda_{+})$ are the two possible signatures. Then, we consider 
the (unordered) signature partition $\mathcal{S}(\mathcal{U})=\{\Lambda_{+}, \Lambda_{-}\}$.
Given a subset $I\subset \{1,\dots,n\}$, with $|I|=d+2$, and $A$ and $C$ uniform, condition (1) in 
Definition~\ref{def:compatible} is equivalent to the following condition:
\begin{enumerate}
\item[(1')] $\mathcal{S}(\mathcal{A}_{I})=\mathcal{S}({\mathscr{C}}_{I})$
\end{enumerate}
In this case, condition (3) in Definition~\ref{def:compatible} implies that 
$\{a_{j}\}\cup\mathcal{A}_{J}$ and  $\{C_j\}\cup{\mathscr{C}}_{J}$ have the same 
signature partition, which is $(d+1,1)$.
\end{remark}

\subsection{{The case $k=2$}\label{ssec:k=2}} 
 
The point configurations such that the corresponding matrix admits a Gale dual which is 
dominating are limited. So, if we are not in this case, checking condition (1) of Theorem~\ref{basicmainresult} 
involves knowing the incidences of the facets of the polytope $\Delta_P$.
However, we now show that in case $A$ has codimension $k=2$,  
there always exists a choice of Gale dual matrix  $B$ such that we can conclude that $n_{\mathcal A}(C)>0$ with the help
of Theorem~\ref{basicmainresult} without checking Condition (1) as  it becomes a consequence of the other conditions.

\begin{lemma}\label{lemma:k=2,B, bis} Assume that $A$ and $C$ are uniform matrices and $k=2$. Suppose that ${\bf 0}\in\mathcal C^\circ$.
Then there exists a matrix $B$ Gale dual to $A$ such that for any matrix $D$ Gale dual
 to $C$ for which $\Delta_P$ is nonempty, bounded and the condition (2) of 
 Theorem~\ref{basicmainresult} is satisfied, condition (1) of Theorem~\ref{basicmainresult} is satisfied, and thus $n_{\mathcal A}(C)>0$.
\end{lemma}
\begin{proof}
Let $B$ be any Gale dual matrix of $A$  with row vectors $b_{1}, \ldots, b_{{n}}$. 
Choose any $i_{2 }\in I_{C}$. Then there exist $i_{1} \in I_{C}$ such that the cone $\R_{>0}b_{i_{1}}+\R_{<0}b_{i_{2}}$ does not
contain vectors $b_{i}$ with $i \in I_{C}$. Note that the latter cone has dimension two since 
$A$ is uniform (which implies that $B$ is uniform as well).
There exists a matrix $R$ of rank two 
such that $B_{\{i_{1},i_{2}\}} \cdot R= \I_{2}$ (if we assume that
$A$, $B$ have integer entries, then there exists an integer matrix $R$ of rank two such that $B_{\{i_{1},i_{2}\}} \cdot R= a \cdot \I_{2}$
where $a=|\det(B_{\{i,j\}})|$).
Consider the matrix $B'=B \cdot R$, with row vectors $b'_{1},\ldots,b_{n}'$. Then $B'$ is a 
Gale dual matrix to $A$ such that $b'_{i_{1}}=(1,0)$, $b'_{i_{2}}=(0,1)$
and the open quadrant $\R_{>0} \times \R_{<0}$ does not contain any vector $b_{i}'$ with $i \in I_{C}$.
Note also that if $i \in I_{C}$ and $i \neq i_{1}, i_{2}$ then both coordinates of $b_{i}'$ 
are nonzero for otherwise this would give
a vanishing maximal minor of $B'$. In particular, we get $b_{i}' \neq 0$, and thus $b_{i}'$
 is not weakly mixed, for all
$i \in I_{C}$. Suppose now that there are two distinct vectors $b_{i}'$ and $b_{j}'$ with $i,j \in I_{C}$ 
such that the submatrix $B_{\{i,j\}}$ is weakly mixed. Then these row vectors lie in opposite quadrants 
of $\R^{2}$ and these quadrants should be $\R_{>0}^{2}$ and $\R_{<0}^2$. But then the cone 
$\R_{>0}b_{i}'+\R_{>0}b_{j}'$ contains either $b'_{i_{1}}=(1,0)$ or $b'_{i_{2}}=(0,1)$, and thus
$\{i,j\} \notin F_{L}$.
\end{proof}

Given a vector $v\in \R^n$ and $I\subset \{1,\dots,n\}$ we denote by $v_I\in\R^{|I|}$ the vector obtained from $v$ 
after removing the coordinates with indexes that do not belong to $I$. Given a set $S\subset \R^n$, we denote  $S_I=\{v_I\, :\, v\in S\}$.

Consider the four open quadrants of $\R^2$ numbered from $1$ to $4$, where the signs of the
 two coordinates are $(+,+)$, $(-,+)$, $(-,-)$, and $(+,-)$ for the first, second, third and fourth quadrant respectively.
In case that there exists a Gale dual matrix $B$ with rows in each of the quadrants, we have the following result. 

\begin{lemma}\label{lemma:4quadrants} Given $A\in\R^{(d+1)\times (d+3)}$ 
uniform,
let $B\in\R^{(d+3)\times 2}$ be a Gale dual matrix of $A$.
Suppose there exists rows of $B$, $b_{i_j}$, with $1\leq j \leq 4$, such that $b_{i_j}$ lies in the $j$-th open 
quadrant of $\R^2$. Let $C\in\R^{d\times n}$ uniform. Suppose that ${\bf 0}\in {\mathcal C}^{\circ}$. 
Assume moreover that given a Gale dual matrix of $C$, the row vectors $P_{i_1}, \dots, P_{i_4}$ define
normals to facets of the closure of the cone ${\mathcal C}_P^\nu$ in~\eqref{eq:Cnu}.
If $\sign((B_j)_{I_C})\in\sign((\ker (C))_{I_C})$ for $j=1,2$, then $n_{\mathcal{A}}(C)>0$.
\end{lemma}

Note that the condition that $P_{i_1}, \dots, P_{i_4}$
define normals to facets  of  the associated cone ${\mathcal C}_P^\nu$  is independent of the choice of Gale dual matrix of $C$. 

\begin{proof} As $\sign((B_j)_{I_C})\in\sign((\ker (C))_{I_C})$ for $j=1,2$, there are vectors $D_1, D_2\in \ker(C)$ such that 
$\sign((D_j)_{I_C})=\sign((B_j)_{I_C})$ for each $j=1,2$. We can assume that $D_1$ and $D_2$ are 
linearly independent. If not, the zero coordinates of $D_1$ and $D_2$ (which are at most two, since $C$ is uniform)
 have to be the same. That is, $(D_1)_j=0$ if and only if $(D_2)_j=0$ (otherwise, they cannot be linearly dependent). 
 Suppose that $(D_1)_j=(D_2)_j=0$ for certain $j$. If $j\in I_C$, then $(B_1)_j=(B_2)_j=0$, but since $A$ is uniform, 
 $B_1$ and $B_2$ have at most one zero coordinate, and then, $B_1$ and $B_2$ are scalar multiples of each other, 
 a contradiction. Then if $(D_1)_j=(D_2)_j=0$, $j\notin I_C$. We take a vector $v$ in $\ker(C)$ such that $D_1$ 
 and $v$ are linearly independent. Then we can take $D'_2=D_2+\lambda v$, with $\lambda$ small enough 
 such that  $\sign((D_2)_{I_C})=\sign((B_2)_{I_C})$.

So, we can suppose that $D_1$ and $D_2$ are linearly independent. Consider the matrix $\tilde{D}$ 
with column vectors $D_1$ and $D_2$. We have that ${\bf 0}$ belongs to the open cone generated 
by the rows of $\tilde{D}$, because the $i_j$-th row of $\tilde{D}$ belongs to the $j$-th open quadrant,
 then Condition~\ref{condbis} of Lemma~\ref{addpositive vector} is satisfied. As ${\bf 0}\in {\mathcal C}^{\circ}$, 
 by Lemma~\ref{addpositive vector} and Corollary~\ref{lemma: can be made bounded}, there exists a 
 positive vector $D_0$ such that the matrix obtained from $\tilde{D}$ by adding $D_0$ as a first column vector
  is Gale dual to $C$ and the associated polytope $\Delta_P$ is nonempty and bounded. Also note that $\Delta_P$ 
  has a facet for each row vector $i_j$ of $\tilde{D}$, each one in the $j$-quadrant of $\R^2$, for $j=1,\dots,4$. 
  Then, if we have a $2\times 2$ mixed submatrix of $B$, it does not correspond to a submatrix $B_L$, with 
$L\in\mathcal{F}(\Delta_P)$ (and any row of $\tilde{D}$ corresponding to $i\in I_C$ is not equal to zero). 
Then, all the conditions of Theorem~\ref{basicmainresult} are satisfied and $n_{\mathcal{A}}(C)>0$.
 
\end{proof}

\section {Algebraic conditions and real solutions of integer configurations}
\label{sec:integer}

In this section we will consider integer configurations $\mathcal A$ and thus, integer 
matrices $A$.  Interestingly,  in Corollary~\ref{cor:ci} we will relate Lemma~\ref{lemma:nominormixed} with known algebraic results
in the study of toric ideals~\cite[Ch.4]{sturmfels}. Indeed, we summarize in \S~\ref{ssec:alg}
some known algebraic results that show the existence of a  mixed dominating
Gale dual matrix is equivalent to the fact that there is a full dimensional
sublattice of the integer kernel ${\rm ker}_\Z(A)$ whose associated lattice ideal~\eqref{eq:li} is a complete
intersection. This means that it can be generated by as many polynomials as the codimension 
of its zero set. In the opposite spectrum, when an ideal is not Cohen-Macaulay   there is no such direct 
relation between algebra and geometry (see for instance~\cite{Eisenbud}, or Chapter 1 in~\cite{St} in the graded case).
When $k=2$, we
also consider lattice ideals  which are not Cohen-Macaulay. Proposition~\ref{prop:nCM}  shows how to deal 
with this more complicated algebraic case. Also, in \S~\ref{ssec:reals} we naturally extend the search for positive solutions 
to the search for real solutions with nonzero coordinates.

\subsection{Algebraic conditions} \label{ssec:alg}
 A polynomial ideal is called \emph{binomial} if it can be generated with polynomials with
at most two terms. A subgroup $\mathcal{L}\subset \Z^n$ is called a \emph{lattice}. 
We associate to a lattice $\mathcal L$ the following binomial
ideal:
\begin{equation}\label{eq:li}
I_{\mathcal{L}}=\langle x^{u^+}-x^{u^-} \, : \, u\in\mathcal{L} \rangle \subset \R[x_1, \dots,x_n],\end{equation}
 where $u=u^+-{u}^-$ is the decomposition in positive and negative components. 
 For example, if $u=(1,-2,1,0)\in\Z^4$, then $x^{u^+}-x^{u^-}=x_1x_3-x_2^2$. 

Given a configuration 
$\mathcal{A}=\{a_1,\dots,a_{n}\}\subset \Z^d$ 
of integral points, and the associated matrix $A\in\Z^{(d+1)\times n}$, let $B\in \Z^{n\times k}$ be a Gale dual matrix of $A$, and 
denote by $B_{1},\dots,B_k$ the column vectors of $B$. 
Note that $\{B_1,\dots,B_k\}$ is a $\Q$-basis of $\ker_{\Z}(A)$, but it is not necessarily a 
$\Z$-basis unless the greatest common divisor of the maximal minors of $B$ is equal to $1$. 
When this is the case, we will say that $B$ is a $\Z$-Gale dual of $A$.
We associate to any choice of Gale dual $B$ of $A$ the following lattice: 
\[\mathcal{L}_B=\Z B= \Z B_1\oplus\cdots\oplus\Z B_k\subset\Z^n,\] 
and its corresponding lattice ideal $I_{\mathcal{L_B}}$.
In particular,  when $\mathcal{L}_B=\ker_\Z(A)$, then the lattice ideal $I_{\mathcal{L}_B}$ is known as the \emph{toric ideal} $I_A$.
We have the following known result from \cite{FS}. See also Theorem~2.1 
of~\cite{MatusevichSobieska}, where the notation is similar to the notation of this paper.

\begin{thm}[\cite{FS}, Theorem 2.9]\label{th:FS} The lattice ideal $I_{\mathcal{L}_{B}}$ is a complete
intersection if and only if $\mathcal{L}_{B} = \mathcal{L}_{B'}$ for some dominating matrix 
$B'\in\Z^{n\times k}$. In this case, $I_{\mathcal{L}_{B}}=\langle x^{u^+}-x^{u^-} \, : \, u \text{ is a column of } B'\rangle$.
\end{thm}

The following result is a direct consequence of Lemma~\ref{lemma:nominormixed} and Theorem~\ref{th:FS}.

\begin{corollary}\label{cor:ci}
 If $A\in\Z^{(d+1)\times n}$ and $C\in \R^{d\times n}$ are 
 uniform matrices and $B\in\Z^{n\times k}$ is a Gale dual matrix of $A$ such that the lattice ideal $I_{\mathcal{L}_{B}}$ 
is a complete intersection, then there exists a Gale dual  matrix $B'\in\Z^{n\times k}$ of $A$ which satisfies the condition (1) of 
Theorem~\ref{basicmainresult}.
\end{corollary}

Given $A$, let $B\in\Z^{n\times k}$ a Gale dual matrix of $A$, and consider the lattice $\mathcal{L}_B=\Z B$.
The set of rows of $B$, $\{b_1,\dots,b_n\}\subset \Z^k$ is called a Gale diagram of $\mathcal{L}_B$.  
Any other $\Z$-basis for $\mathcal{L}_B$ yields a Gale diagram, which means that Gale diagrams are unique 
up to transformation by an invertible integer matrix.

The following proposition from~\cite{PeevaSturmfels} relates Gale diagrams with algebraic 
properties of the lattice ideal $\mathcal{L}_{ B}$ when $k=2$:

\begin{prop}[\cite{PeevaSturmfels}, Proposition 4.1]\label{th:latticek=2} Given $A\in\Z^{(d+1)\times (d+3)}$, 
let $B\in\Z^{n\times 2}$ be a Gale dual matrix of $A$.
The lattice ideal $I_{\mathcal{L}_B}$ is not Cohen-Macaulay if and only if it has a Gale diagram which
intersects all the four open quadrants of $\R^2$.
\end{prop}

The following result follows from Proposition~\ref{th:latticek=2} and Lemma~\ref{lemma:4quadrants}.

\begin{prop}\label{prop:nCM} Given $A\in\Z^{(d+1)\times (d+3)}$ uniform, let $B\in\Z^{n\times 2}$ 
be a Gale dual matrix of $A$. Suppose that the lattice ideal $I_{\mathcal{L}_B}$ is not 
Cohen-Macaulay and let $B'$ be any other Gale Dual matrix of $A$ such that the columns $B'_1, B'_2$
of $B'$ form a $\Z$-basis of $\mathcal{L}_B$ and such that the corresponding Gale diagram 
$\{b'_1,\dots,b'_n\}$ intersects all the four open quadrants of $\R^2$. Let $b'_{i_j}$, 
with $1\leq j \leq 4$,  be rows of $B'$ each lying in the interior of a different
open quadrant in $\R^2$.
 Let $C\in\R^{d\times n}$ uniform satisfying ${\bf 0}\in {\mathcal C}^{\circ}$. Assume moreover
  that given a Gale dual matrix of $C$, the row vectors $P_{i_1}, \dots, P_{i_4}$ define
 normals to facets of the closure of the cone ${\mathcal C}_P^\nu$ in~\eqref{eq:Cnu}.
 
 Then, if $\sign((B'_j)_{I_C})\in\sign((\ker (C))_{I_C})$ for $j=1,2$, then $n_{\mathcal{A}}(C)>0$.
\end{prop}

\subsection{Real solutions}\label{ssec:reals}

When $A$ has integer entries,~\eqref{E:system} is a system of Laurent polynomials with real coefficients, which are defined
over the real torus $(\R^*)^d$. 
In this subsection, we are interested on the existence of real  solutions of ~\eqref{E:system} 
with nonzero coordinates for integer matrices $A$ of
exponents.  Our main result is Theorem~\ref{thm:positive impliesreal}.

Given any $s=(s_1,\ldots,s_d) \in \Z^d$, denote by $\R^d_{s}$ the orthant 
$$\R^d_{s} \, = \, \{x \in \R^d \, : \, (-1)^{s_i}x_i >0, i=1,\ldots,d\}.$$
In particular, $\R^d_{s}=\R_{>0}^d$ if $s \in 2\Z^d$.
Let $x \in (\R^*)^d$ be a solution of ~\eqref{E:system}. Then $x \in \R^d_s$ for some 
$s  \in \Z^d$ (which is unique up to adding a vector in $2\Z^d$).
Setting $z_i=(-1)^{s_i}x_i$, we get that $z=(z_1,\ldots,z_d)$ is a positive solution of the 
system with exponent matrix $A$ and coefficient matrix $C_s$ defined by $(C_s)_{ij}= (-1)^{\langle s, a_j\rangle}c_{ij}$. 
Moreover, if $D$ is a Gale dual matrix of $C$, then the matrix $D_s$ defined by
$(D_s)_{ij}=(-1)^{\langle s, a_i\rangle}d_{ij}$ is a Gale dual matrix of $C_s$. Denote by $P_{i,s}$ the $i$-th row vector of $D_s$.
Thus $P_{i,0}=P_i$ ($i$-th row of $D$) and $P_{i,s}=(-1)^{\langle s, a_i \rangle}P_{i}$, $i=1,\ldots,n$. Denote by
${{\mathcal C}}_{P_s}$ the positive cone  generated by $P_{i,s}$ 
for $i=1,\ldots,n$.
%

Let ${\mathcal M}_P$ denote  the complement in $\R^{k+1}$ of the hyperplane arrangement given 
by the hyperplanes $\{y \in \R^{k+1} \, : \, \langle P_i, y \rangle=0\}$, $i=1,\ldots,n$. 
For any $\varepsilon \in \Z^{n}$ denote by  ${\mathcal C}^{\nu}_{\varepsilon}$ the connected component of 
${\mathcal M}_P$ defined by 
$${\mathcal C}_{\varepsilon}^\nu \, = \, \{ y \in \R^{k+1} \, : \, (-1)^{\varepsilon_i} \langle P_i,y \rangle >0, \, i=1,\ldots, n\}.$$
 Note that ${\mathcal C}_0^\nu={\mathcal C}_{P}^{\nu}$.
 
Write $A'$ for the matrix with column vectors $a_1,\ldots,a_{n}$ ($A'$ is obtained by removing the first row of $A$). 
 It is convenient to introduce the map
$\psi:  \Z^{d} \rightarrow  \Z^{1 \times n}$ defined by 
\begin{equation}\label{def:psimap}
\psi(s)=s \cdot A',
\end{equation}
  where $s \in \Z^{d}$ is considered a
 row vector, i.e, an element of $\Z^{1 \times d}$. 
Then, for any integer vector $b \in \ker(A)$ we have:
$$\prod_{i=1}^{n}\langle P_{i,s}, y \rangle^{b_i}= (-1)^{\langle s, A'b \rangle} \prod_{i=1}^{n} 
\langle P_i, y \rangle^{b_i}=\prod_{i=1}^{n} \langle P_i, y \rangle^{b_i}.$$
Thus, applying Theorem \ref{Galesystems} to the system with coefficient matrix $C_s$ and exponent matrix $A$,
we obtain that the real solutions of ~\eqref{E:system} contained in the orthant $\R^d_s$ are in bijection 
with the solutions of \eqref{homogeneous Gale system} in the quotient 
$\P{\mathcal C}^{\nu}_{\psi(s)}$ of the open cone $\mathcal C_{\psi(s)}^\nu$
by the equivalence relation $\sim$   defined in Section~\ref{sec:back}:
$y \sim y'$ if and only if there exists $\alpha >0$ 
such that $y= \alpha y'$.
In fact, such a bijection is given by the map which associates to any solution
$x \in \R^d_{s}$ of the system~\eqref{E:system} the unique $y\in \R^{k+1}$ such that $x^{a_i}=\langle P_i, y \rangle$ for $i=1,\ldots,n$.
We have proven the following result:

\begin{prop}\label{Galesystemsreal}
For any $s \in \Z^d$, there is a bijection between the real solutions of~\eqref{E:system} contained in $\R_s^d$
and the solutions of~\eqref{homogeneous Gale system} in $\P{\mathcal C}^{\nu}_{\psi(s)}$. This induces
a bijection between the solutions of
~\eqref{E:system} in $\R_s^d$ and the solutions of  \eqref{homogeneous Gale system} 
in $\Delta_{P_s}={\mathcal C}^{\nu}_{\psi(s)} \cap \{y_0=1\}$ when $(1,0,\ldots,0)$ lies in the closure of the cone ${{\mathcal C}}_{P_s}$.
\end{prop}

If $M$ is any matrix or vector with integer entries, we denote by $[M]_2$ the matrix or vector 
with coefficients in the field $\Z / 2 \Z$ obtained by taking the image of each entry by the 
quotient map $\Z \rightarrow \Z / 2 \Z$.  
Note that the following relation between the ranks holds: ${\rm rk}([A]_2)={\rm rk}([A']_2)$  if $[(1,1,\ldots,1)]_2$
belongs to the row span of $[A']_2$ and ${\rm rk}([A]_2)={\rm rk}([A']_2)+1$ otherwise.
The following result is straightforward.

\begin{lemma}\label{L:image}
For any $s,s' \in \Z^d$, we have
${\mathcal C}^{\nu}_{\psi(s)}={\mathcal C}^{\nu}_{\psi(s')}$ if and only if $[s'-s]_2$ belongs to the left kernel of $[A']_2$.
For each $s \in \Z^d$, there are  $2^{d-\rk([A']_2)}$ distinct orthants $\R^d_{s'}$ such that 
${\mathcal C}^{\nu}_{\psi(s)}={\mathcal C}^{\nu}_{\psi(s')}$.
\end{lemma}

Recall that since $A$ contains a row of ones, each polynomial in \eqref{homogeneous Gale system} 
is homogeneous of degree $0$, which implies the following fact.

\begin{lemma} \label{pair}
For any $\varepsilon \in \Z^{n}$, the map $y \mapsto -y$ induces a bijection between the solutions 
of \eqref{homogeneous Gale system} in ${\mathcal C}^\nu_{\varepsilon}$ and the  solutions of 
\eqref{homogeneous Gale system} in ${\mathcal C}^{\nu}_{\varepsilon+(1,1,\ldots,1)}$.
\end{lemma}

Choose a $\Z$-Gale dual matrix $B \in \Z^{n \times k}$  of $A$ and 
consider the Gale dual system~\eqref{homogeneous Gale system} it defines, for a given Gale dual matrix $D$ 
of a full rank matrix $C$. Since $B$ is a $\Z$-Gale dual matrix of $A$, we have $[A]_{2} \cdot [B]_{2}=0$ and furthermore the column vectors of 
$[B]_{2}$ are linearly independent over $\Z / 2\Z$ because the greatest common divisor of the maximal 
minors of $B$ is equal to $1$. Therefore, $[B]_{2}$ is a Gale dual matrix of $[A]_{2}$ if $\rk([A]_2)=d+1$.
If the corank $c=d+1-\rk([A]_2)$ is positive, then a Gale matrix of $[A]_{2}$ can be obtained by adding $c$ column vectors to
$[B]_{2}$.

Define the \emph{feasible} set $\Omega$ as follows: 
$$\Omega=\{\varepsilon \in \Z^n  \, :  \, \langle [\varepsilon]_{2}, \beta \rangle=0 \mbox{\, for  all \,} \beta \in \mbox{Ker} ([A]_{2})\}.$$
Equivalently, $\Omega$ is the set of all $\varepsilon \in \Z^n$ such that $[\varepsilon]_{2}$ belongs to the left kernel of any Gale dual matrix of $[A]_{2}$.

\begin{prop}\label{twocases}
Let $\varepsilon \in \Z^{n}$. If $\varepsilon=\psi(s)$ for some $s \in \Z^d$, then  $\varepsilon \in \Omega$. 
Conversely, assume that \eqref{homogeneous Gale system} has a solution in ${\mathcal C}_{\varepsilon}^\nu$.
If $\rk([A]_2)=d+1$ then $\varepsilon \in \Omega$. Assume
that $\varepsilon \in \Omega$ if $c=d+1-\rk([A]_2)>0$. Then, the following holds.
\begin{enumerate} 
\item If $\rk([A]_2)=\rk([A']_2)$,  there exists $s \in \Z^{d}$ such that $[\varepsilon]_2= [\psi(s)]_2$.
\item If $\rk([A]_2)>\rk([A']_2)$, either there exists an integer vector $s \in \Z^{d}$ such that $[\varepsilon]_2= [\psi(s)]_2$,
or there exists $s \in \Z^d$  such that $[\varepsilon]_2+[(1,1,\ldots,1)]_2= [\psi(s)]_2$. 
Moreover, there do not exist $s,s' \in \Z^{d}$ such that $[\psi(s')]_2=[(1,1,\ldots,1)]_2+[\psi(s)]_2$, 
so that only one of the two previous cases occurs.
\end{enumerate}
\end{prop}

\begin{proof}
Let $y$ be a solution of \eqref{homogeneous Gale system} such that $y  \in {\mathcal C}^{\nu}_{\varepsilon}$. Then, there exist
positive real numbers $d_{i}$ such that $\langle P_i,y \rangle = (-1)^{\varepsilon_i}d_{i}$ for $i=1,\ldots,n$ and 
using \eqref{homogeneous Gale system} this gives that $\sum_{i=1}^{n} \varepsilon_i b_{ij}$ is an even integer
 number for $j=1,\ldots,k$. Therefore, $[\varepsilon]_{2}$ belongs to the left kernel of $[B]_{2}$, which means 
 that $\epsilon \in \Omega$ when $\rk([A]_2)=d+1$. Assume that $\varepsilon \in \Omega$ if $c=d+1-\rk([A]_2)>0$. 
 Then, we get that $[\varepsilon]_{2}$ belongs to the left kernel of any Gale dual matrix $B_{2}$ for $[A]_{2}$.
This left kernel is the image of the map
 $\Z^{d+1} \rightarrow  {(\Z/ 2\Z)}^{1 \times n}$ 
sending $(s_{0},s_{1},\ldots,s_{d})$ to $[(s_{0},s_{1},\ldots,s_{d})]_2 \cdot [A]_2=[s_{0}(1,1,\ldots,1)]_2+[\psi(s)]_2$, 
where $s=(s_{1},\ldots,s_{d})$. The image of this map is equal to the image of the map $s \mapsto [\psi(s)]_2$ 
precisely when $\rk([A]_2)=\rk([A']_2)$. Thus item (1) and the first part of item (2) are proved. We also get that if
$\epsilon=\psi(s)$ then $[\epsilon]_{2}$ belongs to the left kernel of $B_{2}$, and thus $\varepsilon \in \Omega$.
Finally, note that if $\rk([A]_2)>\rk([A']_2)$ then there do not exist distinct $s,s' \in \Z^{1 \times (d)}$ such that $[\psi(s')]_2=[(1,1,\ldots,1)+[\psi(s)]_2$  for otherwise $[(1,\ldots,1)]_2$ would belong to the row span of $[A']_2$.
\end{proof}

\begin{example}
If $a_{1},\ldots,a_{n} \in 2\Z^{d}$, then $\rk([A']_2)=0$ and $\rk([A]_2)=1$.
Moreover, the number of real solutions of ~\eqref{E:system} is $2^d$ times its number of positive solutions. This number 
is equal to the number of solutions of \eqref{homogeneous Gale system} 
in the quotient cone $\P{\mathcal C}^{\nu}_0=\P{\mathcal C}_{P}^{\nu}$  by Theorem \ref{Galesystems}.
\end{example}

As a direct consequence of Proposition~\ref{Galesystemsreal}, Proposition~\ref{twocases}, Lemma~\ref{L:image}
and Lemma~\ref{pair}, we get the following result.

\begin{thm}\label{thm:positive impliesreal}
If $\rk([A]_2)=d+1$, then there exists a solution of~\eqref{E:system}  in $(\R^*)^d$
 if and only if there exists a solution of~\eqref{homogeneous Gale system} in ${\mathcal M}_P$.
 More generally, when $\rk[A]_2$ is not necessarily equal to $d+1$, there exists a solution of~\eqref{E:system}  in $(\R^*)^d$
if and only if there exists a solution of~\eqref{homogeneous Gale system} in a cone ${\mathcal C}_{\varepsilon}^\nu$ with
$\varepsilon \in \Omega$. 

Moreover, the number of solutions of~\eqref{E:system}  in $(\R^*)^d$ is equal to
$2^{d-\rk([A]_2)}$ times the sum of the number of solutions of~\eqref{homogeneous Gale system} in the 
quotient cones $\P{\mathcal C}^{\nu}_{\varepsilon}$,  over all $\varepsilon \in \Omega$.

 \end{thm}
 
Recall that if $\varepsilon \in \Omega$, then $\varepsilon +(1,\ldots,1) \in \Omega$ because $(1,\dots,1)$ is assumed to be in the row span of $A$ (in fact,
we are assuming that this is a row of $A$).
Thus, the total number of solutions
of~\eqref{homogeneous Gale system} in the quotient cones $\P{\mathcal C}^{\nu}_{\varepsilon}$ over all 
$\varepsilon \in \Omega$ is even due to Lemma~\ref{pair} (the solutions come in pairs of opposite real numbers). 

Therefore, the number of solutions of~\eqref{E:system} in $(\R^*)^d$ is always even if $\rk([A]_2) \leq d$ by Theorem~\ref{thm:positive impliesreal}.
But note that when $\rk([A]_2)=d+1$,  then $2^{d-\rk([A]_2)} = \frac 1 2$,  and so  in this case 
 the number of solutions of~\eqref{E:system} in $(\R^*)^d$ is equal to half the sum of the number of solutions 
 of~\eqref{homogeneous Gale system} in the quotient cones $\P{\mathcal C}^{\nu}_{\varepsilon}$  over all $\varepsilon \in \Z^n$. 
 In fact, by the second item in Proposition \ref{twocases}, when $\rk([A]_2)=d+1$  there is a bijection between the 
 solutions of~\eqref{E:system} in $(\R^*)^d$ and the solutions of~\eqref{homogeneous Gale system} outside the 
 hyperplanes $\{y \in \R^{k+1} \, : \, \langle P_i, y \rangle=0\}$ in the real projective space $\R P^{k}$. 
 The latter result is a consequence of~\cite[Theorem 2.1]{BS08}.

%
%

Given $A$ and $C$ and a choice of Gale dual matrices $B,D$, we saw in the proof of Theorem~\ref{basicmainresult}, 
that under the hypotheses of the theorem, it follows from  Theorem~\ref{Galesystems} that $n_{\mathcal A}(C) >0$ is indeed
equivalent to the existence of a solution to~\eqref{nothomogeneous Gale system} 
in  $\Delta_P$.
In the previous sections, we have given different sufficient conditions on $D$ and $B$ such that  system 
\eqref{nothomogeneous Gale system} has at least one solution in $\Delta_P$. 
When $A$ has integer entries it is then enough to check if these sufficient conditions are satisfied by $B$ and 
any matrix $D_{\varepsilon}$ obtained by multipliying the $i$-th row of $D$ by $(-1)^{\varepsilon_i}$ for 
some  $\varepsilon \in  \Omega$. In this case,  \eqref{E:system}  has at least one solution
 in $(\R^*)^d$ by  Theorem \ref{thm:positive impliesreal}.

\section*{Acknowledgement}
We are grateful to Alessandro A. Grande for raising a question
that lead to the improvement of Section~\ref{sec:main}.

\providecommand{\bysame}{\leavevmode\hbox to3em{\hrulefill}\thinspace}
\providecommand{\MR}{\relax\ifhmode\unskip\space\fi MR }
\providecommand{\MRhref}[2]{%
  \href{http://www.ams.org/mathscinet-getitem?mr=#1}{#2}
}
\providecommand{\href}[2]{#2}

\end{document}